\documentclass[a4paper,10pt]{article}
\usepackage[T1,T2A]{fontenc}
\usepackage[cp1251]{inputenc}
\usepackage[english]{babel}
\usepackage{amsmath}
\usepackage{amsfonts}
\usepackage{amssymb}
\usepackage{mathrsfs}
\usepackage{amsthm}
\usepackage[normalem]{ulem}
\usepackage{titling}
\usepackage{multicol}
\usepackage{graphicx}
\usepackage{wrapfig}
\usepackage{appendix}
\graphicspath{{Pictures/}}

\usepackage{color}
\definecolor{mypink1}{rgb}{0.858, 0.188, 0.478}

\newcommand{\stkout}[1]{\ifmmode\text{\sout{\ensuremath{#1}}}\else\sout{#1}\fi}
\usepackage{euscript}

\newcommand{\F}{\mathcal{F}}
\newcommand{\om}{\vartheta}

\newcommand{\Gw}{R}

\newcommand{\df}{\buildrel\mathrm{def}\over=}
\newcommand{\Bell}{\boldsymbol{B}}
\newcommand{\BU}{\widehat{B}}

\newcommand{\BMO}{\mathrm{BMO}}
\newcommand{\BMOm}{\mathrm{BMO}^{\mathrm{m}}}
\newcommand{\eps}{\varepsilon}
\newcommand{\vf}{\varphi}

\DeclareMathOperator{\conv}{conv}
\DeclareMathOperator{\E}{\mathbb{E}}

\renewcommand{\le}{\leqslant}
\renewcommand{\ge}{\geqslant}
\renewcommand{\leq}{\leqslant}
\renewcommand{\geq}{\geqslant}
\renewcommand{\emptyset}{\varnothing}

\newcommand{\xx}{\bar{x}}
\newcommand{\yy}{\bar{y}}
\newcommand{\zz}{\bar{z}}
\newcommand{\cc}{\bar{c}}

\newcommand{\mmm}{{\bf M}}
\newcommand{\ttt}{\rho}
\newcommand{\bbb}[2]{b_{{}_{\sqrt{#2-#1^2}}}(#1,#2)}
\newcommand{\rrr}{{\bf r}}

\newcommand{\OmSmile}{\omega_{\mathrm{smile}}}
\newcommand{\OmR}{\omega_{{}_\mathrm{R}}}
\newcommand{\OmL}{\omega_{{}_\mathrm{L}}}

\newcommand{\R}{\mathbb{R}}

\newcommand{\eq}[1]{\begin{equation}{#1}\end{equation}}
\newcommand{\mlt}[1]{\begin{multline}{#1}\end{multline}}
\newcommand{\alg}[1]{\begin{align}{#1}\end{align}}

\newcommand{\Bs}{{\bf R}}%соответствует модельной функции Беллмана в улыбке 

\newcommand{\Set}[2]{\Big\{{#1}\;\Big|\,{#2}\Big\}}

\newcommand{\uR}{u_{{}_{\mathrm{R}}}}
\newcommand{\uL}{u_{{}_{\mathrm{L}}}}
\newcommand{\mr}{m_{{}_{\mathrm{R}}}}
\newcommand{\ml}{m_{{}_{\mathrm{L}}}}

\newcommand{\GL}{G_{{}_{\mathrm{L}}}}
\newcommand{\GR}{G_{{}_{\mathrm{R}}}}
\newcommand{\vr}{v_{{}_{\mathrm{R}}}}
\newcommand{\ur}{u_{{}_{\mathrm{R}}}}
\newcommand{\ul}{u_{{}_{\mathrm{L}}}}
\newcommand{\vl}{v_{{}_{\mathrm{L}}}}

\newcommand{\fff}{\Phi}
\newcommand{\uuu}{v}
\newcommand{\vvv}{w}

\newcommand{\Geqref}[1]{\stackrel{\scriptscriptstyle{\eqref{#1}}}{\geq}}

\newcommand{\tl}{t_{\mathrm{L}}}
\newcommand{\tr}{t_{\mathrm{R}}}

\newtheorem{Le}{Lemma}[section]
\newtheorem{Def}[Le]{Definition}

\newtheorem{Th}[Le]{Theorem}
\newtheorem{Cor}[Le]{Corollary}
\newtheorem{Rem}[Le]{Remark}

\numberwithin{equation}{section}

\begin{document}
\author{Dmitriy~Stolyarov \and Vasily~Vasyunin \and Pavel~Zatitskiy \and Ilya~Zlotnikov}
\title{Sharp moment estimates for martingales 
\\
with uniformly bounded square functions
\thanks{Supported by the Russian Science Foundation Grant 19-71-10023.}}

\maketitle
\begin{abstract}
We provide sharp bounds for the exponential moments and~$p$-moments,~$1\leq p \leq 2$, of the 
terminate distribution of a martingale whose square function is uniformly bounded by one. 
We introduce a Bellman function for the corresponding extremal problem and reduce it to 
the already known Bellman function on~$\BMO([0,1])$. In the case of tail estimates, a similar 
reduction does not work exactly, so we come up with a fine supersolution that leads to sharp tail estimates.
\end{abstract}

\section{Introduction}\label{S1}
\subsection{Chang--Wilson--Wolff inequality}\label{s11}

Let~$(X,\Sigma,P)$ be an atomless complete probability space equipped with a discrete time 
filtration~$\F = \{\F_n\}_{n\geq 0}$. Let~$\F_0 = \{\varnothing, X\}$ and let~$\F$ generate~$\Sigma$. 
Assume for simplicity that each~$\sigma$-algebra~$\F_n$ is finite, i.\,e., consists of a finite 
number of sets. Consider a real-valued martingale~$\varphi = \{\varphi_n\}_n$ adapted to~$\F$ 
and define its square function~$S\varphi$ by the formula
\begin{equation}
S\varphi = \Big(\sum\limits_{n=0}^\infty (\varphi_{n+1} - \varphi_n)^2\Big)^\frac12.
\end{equation}
In what follows we will always talk about real martingales adapted to filtrations as above unless otherwise specified. 
We call a martingale~$\varphi$ simple if $\varphi_{n+1}=\varphi_n$ for $n$ sufficiently large.
In this paper, we make an attempt to describe the distribution of~$\varphi_{\infty}$ (which is 
the limit value of the martingale,~$\varphi_{\infty} = \lim_{n\to \infty} \varphi_n$) under 
the assumption that~$S\varphi$ is uniformly bounded. From general theory (see~\eqref{Orthogonality} below),~$\varphi$ is 
a~$\BMO$-martingale provided~$S\varphi \in L_{\infty}$. Thus, by the John--Nirenberg 
inequality,~$\varphi_{\infty}$ is a subexponential random variable. Namely, there exist 
positive constants~$c_1$ and~$c_2$ such that 
\begin{equation}\label{WeakJNNonSharp}
P(\varphi_\infty -\varphi_0 \geq t)\leq c_2e^{-\frac{c_1t}{\ \|S\varphi\|_{L_{\infty}}}}, \qquad t>0,
\end{equation}
for any martingale~$\varphi$. We focus on sharp estimates of this kind. In particular, we aim 
to compute the best possible values of~$c_1$ and~$c_2$ (see Corollary~\ref{cor_sharp_tail} below). 
According to the knowledge of the authors, such sharp estimates are not known. 

In the case where~$\F$ is a dyadic filtration (by that we mean that any atom in~$\F_n$ is split 
into two atoms of equal mass in~$\F_{n+1}$), a much better estimate exists. The famous 
Chang--Wilson--Wolff inequality (see Theorem 3.1 in~\cite{CWW1985} for the original formulation and~\cite{SlavinVolberg2007} for further development) says that the 
distribution of~$\varphi$ is subgaussian:
\begin{equation}
P(\varphi_{\infty} -\varphi_0 \geq t)\leq e^{-\frac{t^2}{\ 2\|S\varphi\|^2_{L_{\infty}}}}, \qquad t>0.
\end{equation}
In a recent paper~\cite{IvanisviliTreil2019}, Ivanisvili and Treil generalized this result to the 
case where the filtration~$\F$ has bounded distortion~$\alpha$, which means that each atom in~$\F_n$ 
has at least~$\alpha$ mass of its parental atom. In this case,
\begin{equation}
P(\varphi_{\infty} -\varphi_0 \geq t)\leq e^{-\frac{\alpha t^2}{\ \|S\varphi\|^2_{L_{\infty}}}}, \qquad t>0.
\end{equation}
This result hints us that the distribution of~$\varphi_{\infty}$ may no longer be subgaussian if we 
do not make assumptions about regularity of the filtration. As we will see later, this is indeed 
the case (for example, the inequality~\eqref{WeakJNNonSharp} is sharp for certain choice of~$c_1$ 
and~$c_2$, see Corollary~\ref{cor_sharp_tail} below). 

Since we focus on sharp estimates, it is natural to consider not only tail estimates, but also 
inequalities for the exponential moments and~$p$-moments. In particular, one may wonder what are 
the largest possible values of the quantities~$\E e^{\lambda \varphi_\infty}$, or~$\E |\varphi_\infty|^p$ 
under the assumption~$\|S\varphi\|_{L_{\infty}} \leq 1$. We will partially answer this question, 
see Corollaries~\ref{p-momentCorollary} and~\ref{exponentialmomentCorollary} below.  One may go further, 
pick an arbitrary function~$f$, and ask about the largest possible value of~$\E f(\varphi_\infty)$ under 
the same assumption. We will study this problem for the cases when~$f'''$ either does not change sign 
or changes sign from~$+$ to~$-$ once.

Some of the results of the present paper were announced in the short report~\cite{SVZZ2019}. 
We also provided some proofs there. The present paper contains the remaining proofs. 
In a sense,~\cite{SVZZ2019} contains the reasoning that do not depend on the geometry of specific 
Bellman functions. They are much shorter than the treatment of Bellman functions we present here.

For the reader who is not interested in the Bellman function technique, 
Corollaries~\ref{p-momentCorollary},~\ref{exponentialmomentCorollary}, and~\ref{cor_sharp_tail} 
may be considered as the main results of the paper. Lemma~\ref{SimpleMajorant} and 
Theorems~\ref{CoincidenceTheorem} and \ref{th_bellman_is_less_b_eps} are more important 
from the Bellman function point of view.

\subsection{Estimates for $\BMO$ functions}\label{s12}
The space~$\BMO$ is pivotal for our considerations. There are several equivalent norms in this space. 
Since we are dealing with sharp estimates, the choice of a specific norm is crucial. 
The space~$\BMOm$ called the space of martingales of bounded mean oscillation is defined 
as follows (see, e.\,g., Chapter II in~\cite{Kazamaki1994})
\begin{equation*}
\|\varphi\|_{\BMOm}^2 = 
\sup\Big\{\big\|\E\big((\varphi_{\infty} - \varphi_{\tau})^2\mid \F_{\tau}\big)\big\|_{L_\infty}\,\Big|\;
\tau \text{ is a stopping time}\Big\}.
\end{equation*}
A simple orthogonality argument
\begin{multline}\label{Orthogonality}
\E\big((\varphi_{\infty} - \varphi_{\tau})^2\mid \F_{\tau}\big) = 
\E\Big(\Big(\sum\limits_{n \geq \tau} (\varphi_{n+1} -\varphi_n)\Big)^2\,\Big|\; \F_{\tau}\Big) =
\\ 
\E\Big(\sum\limits_{n \geq \tau}(\varphi_{n+1} -\varphi_n)^2\,\Big|\; \F_{\tau}\Big) \leq 
\E\big((S\varphi)^2\mid \F_\tau\big),
\end{multline}
leads to the inequality~$\|\varphi\|_{\BMOm} \leq \|S\varphi\|_{L_{\infty}}$. 

The space~$\BMOm$ has its real analysis counterpart (see, e.\,g., Chapter IV in~\cite{Stein1993} 
for more information). The~$\BMO$ space on the unit interval is defined with the help of the seminorm
\begin{equation}\label{BMOnormDefinition}
\|\psi\|_{\BMO([0,1])}^2=
\sup\Big\{\frac{1}{|J|}\int\limits_{J}\Big(\psi(x) - \frac{1}{|J|}\int\limits_J \psi\Big)^2dx\;\Big|\, 
J \text{ is a subinterval of } [0,1]\Big\}.
\end{equation}
We note that this definition is not the most common in real analysis. A version based on the~$L_1$ 
norm instead of~$L_2$ is more widespread (the two~$\BMO$ seminorms are equivalent). The~$L_2$ based 
version is closely related to the martingale~$\BMOm$ space. We denote the non-increasing rearrangement 
(the inverse function to the distribution function of~$\xi$) of a random variable~$\xi$ by ~$\xi^*$:
\begin{equation*}
\xi^*(t) = \inf\{\alpha\mid P(\xi > \alpha) \leq t\},\qquad t \in [0,1).
\end{equation*}

\begin{Th}\label{Embedding}
The inequality~$\|\varphi_{\infty}^*\|_{\BMO([0,1])} \leq \|S\varphi\|_{L_{\infty}}$ holds 
for any martingale~$\varphi$ and is sharp.
\end{Th}

Here and in what follows the notation~$\varphi_{\infty}^*$ means the monotonic rearrangement 
of~$\varphi_\infty$. This theorem was proved in~\cite{SVZZ2019}. Though Theorem~\ref{Embedding} 
says there is a certain relationship between martingales~$\varphi$ whose square function is bounded 
and functions~$\psi$ on the unit interval that belong to the~$\BMO$ space, we warn the reader 
against identification of these classes of objects, which have different nature and origin.

\begin{Rem}
The estimate~$\|\varphi_{\infty}^*\|_{\BMO([0,1])} \leq \|\varphi\|_{\BMOm}$ is not true in general 
for discrete time filtrations. To see that\textup, consider the case where~$\F$ is dyadic. The inequality
\begin{equation*}
\|\varphi_{\infty}^*\|_{\BMO([0,1])} \leq \frac{3}{2\sqrt{2}}\|\varphi\|_{\BMOm},\quad \F \hbox{ is dyadic},
\end{equation*}
is sharp \textup(and true\textup{),} see Corollary~$1$ in~\textup{\cite{SVZ2015}}. What is more\textup, 
for any~$C > 0$ there exist a discrete filtration~$\F$ and a martingale~$\varphi$ adapted to it such 
that the inequality
\begin{equation*}
\|\varphi_{\infty}^*\|_{\BMO([0,1])} \leq C\|\varphi\|_{\BMOm}
\end{equation*}
fails.
\end{Rem}

Theorem~\ref{Embedding} leads to many nice inequalities. In particular, it says that
\eq{\label{Pre-BellmanEmbedding}
\sup\Big\{\E f(\varphi_{\infty})\,\Big|\; \varphi_0 = x,\; \|S\varphi\|_{L_{\infty}} \leq 1\Big\} \leq 
\sup\Big\{\int\limits_0^1 f(\psi)\,\Big|\, \int\limits_0^1\psi = x,\; \|\psi\|_{\BMO([0,1])}\leq 1\Big\}
}
for any non-negative function~$f\colon \R \to \R$. It is reasonable to fix the expectation of 
our martingale since~$\varphi_0$ does not affect the square function, but has strong influence 
on the quantity~$\E f(\varphi_{\infty})$. There are two surprising facts about 
formula~\eqref{Pre-BellmanEmbedding}. The first one is that the inequality turns into equality 
quite often (in particular, for the important cases~$f(t) = e^{\lambda t}$ and~$f(t)=|t|^p,1\leq p\leq 2$). 
The second fact is that the supremum on the right hand side may be computed exactly for arbitrary~$f$, 
which satisfies some mild regularity assumptions. We briefly describe these results.

We fix the second moment as well and write the definition of the Bellman 
function~$b_{\eps}\colon\omega_\eps\to\mathbb{R}$,
\begin{equation}\label{SmallBellman}
\begin{aligned}
b_{\eps}(x,y) &= \sup\bigg\{\int\limits_0^1f(\psi)\,\bigg|\; 
\int_0^1\!\psi = x,\; \int_0^1\!\psi^2 = y,\ \|\psi\|_{\BMO([0,1])}\leq \eps\bigg\}, 
\\ 
\omega_\eps &= \{(x,y)\in\mathbb{R}^2\mid x^2 \leq y\leq x^2+\eps^2\}.
\end{aligned}
\end{equation} 
This Bellman function satisfies the boundary condition~$b_{\eps}(x,x^2) = f(x)$. It appears that 
one may compute the function~$b_\eps$ for arbitrary~$f$. The answer (algorithm) is quite complicated. 
We refer the reader to the paper~\cite{ISVZ2018} for treatment of the general case. 
The paper~\cite{IOSVZ2016} considers less general case (the authors make additional assumptions 
on the structure of~$f$), however, provides a much shorter presentation. The short 
report~\cite{IOSVZ2012} outlines the results. In fact, the particular cases that are important 
for applications were computed in earlier papers~\cite{SlavinVasyunin2011},
\cite{SlavinVasyunin2012},~\cite{VasyuninVolberg2014}, and~\cite{Vasyunin2003}.

The main reason why~$b_{\eps}$ is a tractable object is that it can be described geometrically, 
namely, in terms of locally concave functions. By a locally concave function on a domain 
we mean a function whose restriction to any segment lying in the domain entirely, is concave. 

\begin{Th}[Main theorem and Corollary~5.4 in~\cite{StolyarovZatitskiy2016}]\label{Duality}
Let~$f$ be bounded from below. The function~$b_{\eps}$ can be described as the pointwise minimal 
function among all locally concave functions~$G\colon\omega_\eps\to \mathbb{R}$ that satisfy 
the boundary condition~$G(x,x^2) = f(x)$.
\end{Th}

The fact behind Theorem~\ref{Duality} is that the minimal locally concave function has a good 
probabilistic representation, see Theorem~2.21 in~\cite{StolyarovZatitskiy2016}. We cite a definition 
introduced in~\cite{StolyarovZatitskiy2016} (in fact,~\cite{StolyarovZatitskiy2016} deals with a more 
general situation; in the case of~$\BMO$ and the parabolic strip~$\omega_\eps$ the continuous time 
version of the definition below had appeared in the literature before~\cite{StolyarovZatitskiy2016}, 
see, e.\,g.,~\cite{Kazamaki1994} and~\cite{Osekowski2015}; as the present paper shows, the discrete 
time definition is more convenient in some contexts).

\begin{Def}\label{BasicMartingales}
An~$\mathbb{R}^2$-valued martingale~$M = \{M_n\}_n$ adapted to~$\{\F_n\}_n$ is called 
an~$\omega_{\eps}$-martingale if it satisfies the conditions listed below.
\textup{\begin{enumerate}
\item \emph{$\F_0 = \{\emptyset, X\}$.}
\item \emph{There exists a random variable~$M_{\infty}$ with values 
in~$\{(t,t^2)\mid t\in\mathbb{R}\}$ such that
\begin{equation*}
\E |M_{\infty}| < \infty \quad \hbox{and}\quad M_n = \E(M_{\infty}\mid \F_n).
\end{equation*}}
\item \emph{For every~$n \in \mathbb{Z}_+$ and every atom~$\sigma$ in~$\F_n$
\begin{equation*}
\conv \{M_{n+1}(z)\}_{z\in \sigma} \subset \omega_\eps.
\end{equation*}
}
\end{enumerate}}
\end{Def}

The third requirement should be understood properly: we define~$M_n = \E(M_{\infty}\mid \F_n)$ 
everywhere and thus, consider the convex hull of a finite number of points. By~$\conv A$ we denote 
the convex hull of a set~$A$.

The following lemma plays a crucial role in the proof of Theorem~\ref{Embedding}.

\begin{Le}[Theorem~3.4 in~\cite{StolyarovZatitskiy2016}]\label{MartFunct}
Let~$M$ be an~$\omega_{\eps}$ martingale. The random variable~$M_{\infty}^1$ \textup(the first coordinate 
of the~$\mathbb{R}^2$-valued random variable~$M_{\infty}$\textup)  satisfies the inequality
\eq{
\|(M_{\infty}^1)^*\|_{\BMO([0,1])} \leq \eps.
}
\end{Le}

\subsection{Our results}\label{s13}
The function~$f\colon \R\to\R$ will be subject to some requirements. We will always assume~$f$ 
is measurable and non-negative. Of course, one may use a slightly weaker assumption that~$f$ is 
uniformly bounded from below (or replacing~$f$ with~$-f$, that~$f$ is bounded from above). 
Sometimes we will need a regularity assumption.

\begin{Def}\label{DefStReq} 
We say that~$f$ satisfies the \emph{\bf standard requirements} if it is a non-negative twice 
continuously differentiable function\textup, its third distributional derivative is a signed 
measure\textup, which changes sign only finite number of times\textup, and  
\eq{
\int\limits_{\R}e^{-\frac{|t|}{\eps}}|df''(t)| < \infty\qquad \hbox{for some}\quad \eps > 1.
} 
\end{Def}
These requirements for~$f$ are slightly stronger than in~\cite{ISVZ2018} (the authors of that paper did not require the positivity of~$f$). Note that the 
choices $f(t) = |t|^p$,~$p \in [1,2)$, and $f(t) = \chi_{_{[0,\infty)}}(t)$ do not satisfy 
the standard requirements (the first one is quite close, while the second function is very 
far from being~$C^2$-smooth).

We introduce the main character:
\begin{equation}\label{Bellman}
\Bell(x,y,z) = \sup\Big\{\E H_{f}(\varphi_{\infty},(S\varphi)^2 + z^2)\,\Big|\;\E\varphi_\infty = x,\ 
\E\varphi_{\infty}^2 = y\Big\},\quad z \geq 0,
\end{equation}
where 
\begin{equation*}
H_{f}(s,t) = \begin{cases}
-\infty,\quad &t\notin[0,1];\\
f(s),\quad &t\in [0,1],
\end{cases}
\end{equation*}
and the supremum is taken over all simple martingales $\varphi$ adapted to a discrete time filtration. 
We consider only simple martingales here to avoid technicalities. Note that it suffices to work with 
simple martingales to obtain sharp constants in the inequalities~\eqref{c_p}, \eqref{Exponential}, 
and~\eqref{tail_sq} below. 

This Bellman function will help us to find sharp constants in several inequalities. The reader 
familiar with the Burkholder method (see the original papers~\cite{Burkholder1984} 
and~\cite{NazarovTreil96} or the books~\cite{Osekowski2012},~\cite{VasyuninVolberg2020}) 
may say that the~$y$-coordinate is redundant. However, we prefer to keep it, because it ``tracks'' 
the Hilbert space identities that link the square function to the martingale itself. 

\begin{Rem}\label{Monotonicity}
For any~$x,y$ fixed\textup, the function~$z\mapsto \Bell(x,y,z)$ is non-increasing. This follows 
from formula~\eqref{Bellman}\textup, more specifically\textup, from an equivalent formula
\begin{equation}\label{eq150201}
    \Bell(x,y,z) = \sup\Big\{\E f(\varphi_{\infty})\,\Big|\;\E\varphi_\infty = x,\; 
		\E\varphi_{\infty}^2 = y,\; S\varphi \leq \sqrt{1-z^2}\ \text{a.\,s.}\,\Big\}.
\end{equation}
\end{Rem}
As we will prove a little bit later (see Lemma~\ref{PropertiesOfBellmanFunction} below), 
the natural domain for~$\Bell$ is
\begin{equation}\label{Omega}
\Omega = \Big\{(x,y,z)\in\mathbb{R}^3\,\Big|\;x^2 \leq y \leq 1-z^2 + x^2, \quad z \in [0,1]\Big\}.
\end{equation}

We start with the Bellman function counterpart of Theorem~\ref{Embedding}. Recall 
the function~$b_{\eps}$ defined in~\eqref{SmallBellman}.

\begin{Le}\label{SimpleMajorant}
Let~$f$ be a non-negative function. The inequality~$\Bell(x,y,z) \leq b_{\sqrt{1-z^2}}(x,y)$ 
is true for any triple~$(x,y,z)\in\Omega$.
\end{Le}

This lemma implies~\eqref{Pre-BellmanEmbedding} (plug~$z = 0$ and optimize with respect to~$y$). 
It has already appeared in~\cite{SVZZ2019}. We present its proof in Section~\ref{S2} for completeness 
(in fact, the arguments are quite elementary here).

\begin{Cor}\label{ExtremeSummabilityCor}
Let a measurable function~$h\colon \R \to \R_+$ satisfy
\eq{\label{ExtremeSummability}
\sum\limits_{k\in\mathbb{Z}} e^{-|k|}\sup\limits_{x\in [k-2,k+2]} h(x) < \infty.
}
Then the quantity~$\E h(\varphi_{\infty})$ is finite for any martingale~$\varphi$ such that
$S\varphi\leq1$ almost surely. The bound is uniform with respect to~$\varphi$ as long as~$\varphi_0$ is fixed.
\end{Cor}

This corollary will be proved in Section~\ref{S2}. It is sharp in certain sense. For example, one may 
construct a~$C^3$-smooth function~$h$ such that~$h''' \geq 0$ and~$h(x) = e^x/x$ when~$x$ is 
sufficiently large. Theorem~\ref{CoincidenceTheorem} below then says~$\Bell(x,y,z) = b_{\sqrt{1-z^2}}(x,y)$ 
if both these functions are constructed for~$f := h$. However, one may see that with this function~$h$ 
the Bellman function~$b_1$ is infinite since the integral~$\int_0^1 h(\psi)$ diverges if one plugs
$\psi(t) = -\log t$ into~\eqref{SmallBellman}  (the function~$\log t$ has~$\BMO$-norm equal to one).

As we have said, the inequality in Lemma~\ref{SimpleMajorant} often turns into equality.
\begin{Th}\label{CoincidenceTheorem}
Assume~$f$ satisfies the standard requirements and either~$f''$ is monotone or $f''$ increases up to some 
point and then decreases. Then\textup,~$\Bell(x,y,z) = b_{\sqrt{1-z^2}}(x,y)$ for all~$(x,y,z)\in\Omega$.
\end{Th}

Theorem~\ref{CoincidenceTheorem} was also stated in~\cite{SVZZ2019}, but was not proved. 
Its proof is presented in Subsection~\ref{s330}. Note that particular choices~$f(t) = e^{\lambda t}$ 
and~$f(t) = |t|^p$,~$1\leq p \leq 2$ (this function does not satisfy the standard requirements, however, 
we will be able to cope with this difficulty), fit the assumptions of Theorem~\ref{CoincidenceTheorem}. 
The corresponding functions~$b_{\eps}$ were computed in~\cite{SlavinVasyunin2011} 
and~\cite{SlavinVasyunin2012} respectively. These results will lead us to the corollaries below.

\begin{Cor}\label{p-momentCorollary} 
The optimal constant~$c_p$ in the inequality
\begin{equation}\label{c_p}
\|\varphi_\infty - \varphi_0\|_{L_p}\leq c_p\|S\varphi\|_{L_{\infty}}
\end{equation}
equals~$1$ when~$1 \leq p \leq 2$. 
\end{Cor}

\begin{Cor}\label{exponentialmomentCorollary}
The optimal constant~$C(\eps)$ in the inequality
\begin{equation}\label{Exponential}
\E e^{\varphi-\varphi_0} \leq C(\eps), \quad S\varphi \leq \eps.
\end{equation}
equals~$\frac{e^{-\eps}}{1-\eps}$ when~$\eps < 1$.
\end{Cor}

Sometimes the inequality in Lemma~\ref{SimpleMajorant} is strict on a subdomain of $\Omega$. 
We present the following example corresponding to the choice~$f(t) = \chi_{_{[0,\infty)}}(t)$. 
Note that this function does not fulfill the standard requirements (however, this is not the 
reason for failure of the equality between the Bellman functions; we consider this example 
since it leads to sharp constants in the inequality~\eqref{WeakJNNonSharp}). In this case, 
the function~$b_{\varepsilon}$ was computed in \cite{VasyuninVolberg2014}. 
The domain $\omega_{\varepsilon}$ is split into four parts (see Figure~\ref{fig_Domains_D_eps})
\begin{equation}\label{FlatSplitting}
\begin{aligned}
 D_1^{\eps} &= \{ (x,y) \in \omega_{\eps} \mid  y \ge 2\varepsilon x,  \; x \ge \eps\} 
\cup  \{ (x,y) \in \omega_{\eps} \mid  y \le 2\varepsilon x\};
\\
D_2^{\eps} &=\{(x,y)\in \omega_{\eps} \mid |x|\leq \eps,\; y \geq 2\eps|x|\};
\\
D_3^{\eps} &=\{(x,y)\in \omega_{\eps} \mid y \leq -2 \eps x\};
\\
D_4^{\eps} &=\{(x,y)\in \omega_{\eps} \mid x\leq -\eps,\; y \geq -2 \eps x\},
\end{aligned}
\end{equation}
\begin{figure}
\begin{center}
\includegraphics[width=0.5\columnwidth]{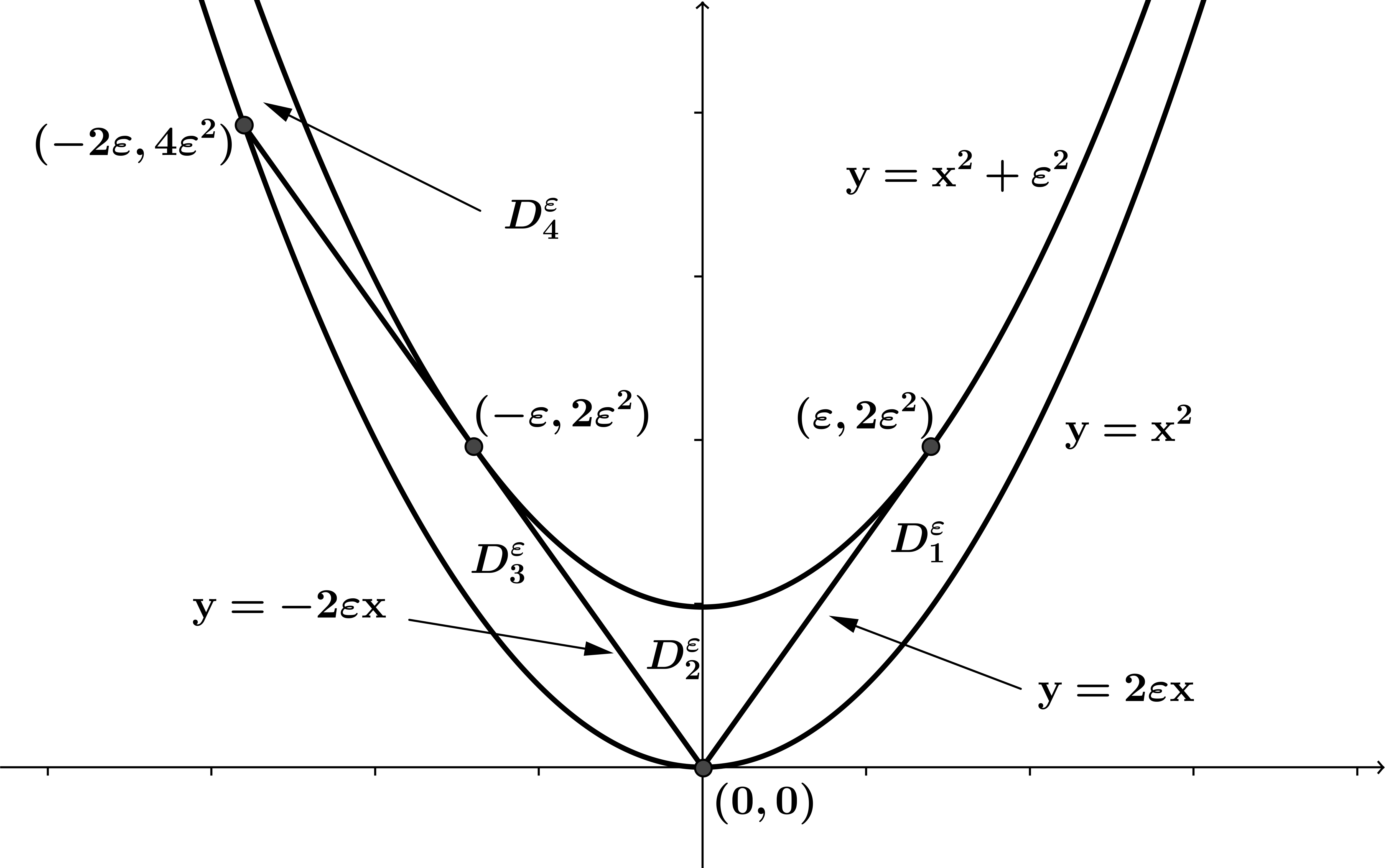}
\caption{Domains $D_j^\eps$}
\label{fig_Domains_D_eps}
\end{center}
\end{figure}
and the function is defined by the formula:
\begin{equation}\label{b_eps_def}
b_{\eps}(x,y) = 
\begin{cases}
\rule{0pt}{15pt}1,\quad & (x,y) \in D_1^{\eps}
\\
\rule{0pt}{15pt}1-\frac{y-2\eps x}{8\eps^2}, \quad &(x,y) \in D_2^{\eps}
\\
\rule{0pt}{15pt}1-\frac{x^2}{y}, \quad &(x,y) \in D_3^{\eps}
\\
\frac{e}{2}\left(1-\sqrt{1-\frac{y-x^2}{\eps^2}}\right)e^{\frac{x}\eps+\sqrt{1-\frac{y-x^2}{\eps^2}}}, 
\quad & (x,y) \in D_4^{\eps}.
\end{cases}
\end{equation}

In Section~\ref{S4}, the function $\Bell$ will be computed on the upper boundary of~$\Omega$, 
namely, we will identify the restriction of~$\Bell$ to\footnote{The subscript~$\mathrm{R}$ in 
the formula below designates the ``roof'' of the domain~$\Omega$.}
\begin{equation}\label{OmegaR}
\Omega_R = \Omega \cap \{z = \sqrt{1 - y + x^2}\}.
\end{equation} 
The set~$\Omega_R$ naturally splits into four parts, each of which is projected onto the corresponding 
domain in~\eqref{FlatSplitting}.

\begin{Th}\label{th_bellman_is_less_b_eps}
Let $f(t) = \chi_{_{[0, \infty)}}(t),$  $t\in\mathbb{R}$. The equality
\begin{equation}
\Bell(x,y,\sqrt{1-y+x^2}) = b_{\sqrt{y-x^2}}(x,y) 
\end{equation} 
holds true whenever~$(x,y)\in D^{\sqrt{y-x^2}}_j$ and $j=1,3,4$. If~$(x,y)\in D_2^{\sqrt{y-x^2}},$ we have 
\begin{equation}
\Bell(x,y, \sqrt{1-y + x^2}) = 1-\frac{\sqrt{1-\ttt^2}-\ttt}{2\sqrt{2}}e^{-\arcsin \ttt-\frac{\pi}{4}}, 
\quad \text{ where }  \ttt = \ttt(x,y) = \frac{x}{\sqrt{2(y-x^2)}}.
\end{equation}
\end{Th}

\begin{Cor}\label{cor_sharp_tail}
The best possible constant $c_1$ in~\eqref{WeakJNNonSharp} equals~$1$. The optimal constant~$c$ 
in the inequality 
\begin{equation}\label{tail_sq}
P(\varphi_\infty-\varphi_0\geq\lambda)\leq c e^{-\frac{\lambda}{\ \|S\varphi\|_{L_{\infty}}}},\qquad\lambda>0,
\end{equation}
equals~$\frac{e}{2}$.
\end{Cor}

Note that the sharp constant~$c$ in the weak type form of the John--Nirenberg inequality
\begin{equation}\label{tail_BMO}
\Big|\big\{t\in[0,1]\ \big|\ \psi(t)-\int_0^1\psi>\lambda\big\}\Big|\leq 
c e^{-\frac{\lambda}{\|\psi\|_{\BMO[0,1]}}}\,,\qquad\lambda>0,
\end{equation}
also equals $\frac{e}{2}$, as it was shown in~\cite{VasyuninVolberg2014}. Even though for this choice 
of~$f$ the inequality in Lemma~\ref{SimpleMajorant} is strict at some points of~$\Omega$, the sharp 
constants in the tail estimates~\eqref{tail_sq} and~\eqref{tail_BMO} for the considered problems coincide.

Though the square function is a very common martingale operator, there are less sharp inequalities 
known about it than about the martingale transform or the maximal function. Even the expression for 
its~$L_p\to L_p$ norm is known only in the range~$p \in (1,2]$ (and in fact, is due to Burkholder 
in~\cite{Burkholder1984}, see Section~$8.3$ in~\cite{Osekowski2012}). The sharp constant in the weak type~$(1,1)$ inequality was found by Cox 
in~\cite{Cox1982} (see also~\cite{Osekowski2014StatisticLetters} for another approach 
and~\cite{Osekowski2009} and~\cite{HIV2020} for related results) while other weak type constants 
are unknown. Sharp inequalities for various special classes of martingales (conditionally symmetric 
martingales, continuous path martingales, etc.) may be found in~\cite{Osekowski2014MAI} and~\cite{Wang1991}. 
We also mention the article~\cite{Osekowski2013}, where questions similar to those considered in the 
present paper are studied in the dyadic setting (namely, that paper studies the distribution of~$S\varphi$ 
under conditions~$\varphi \in L_\infty$ and~$\varphi \in \BMO$ in the dyadic setting). The reader may find 
many interesting sharp inequalities involving the square function in the 8th chapter of~\cite{Osekowski2012}.

In Section~\ref{S2} we study simple properties of the function~$\Bell$ and prove 
Lemma~\ref{SimpleMajorant}, Corollary~\ref{ExtremeSummabilityCor}, and Theorem~\ref{Embedding}. 
Section~\ref{S3} contains the proofs of Theorem~\ref{CoincidenceTheorem}, Corollary~\ref{p-momentCorollary}, 
and Corollary~\ref{exponentialmomentCorollary}. Section~\ref{S4} is devoted to the proofs 
Theorem~\ref{th_bellman_is_less_b_eps} and Corollary~\ref{cor_sharp_tail}.

\section{Main inequality and proof of the majorization theorem}\label{S2}
The lemma we present below is a standard part of the Bellman function method. One may find 
a similar statement in~\cite{Osekowski2012}, see Chapter 8, Theorem~8.1. We provide a proof 
for two reasons: completeness and slight difference between the traditional notation and ours.

\begin{Le}\label{PropertiesOfBellmanFunction}
Let $f \geq 0$. 
\begin{itemize}
\item[\textup{(i)}] The function~$\Bell$ is non-negative on the domain~$\Omega$ defined 
by~\eqref{Omega} and equals~$-\infty$ outside it.
\item[\textup{(ii)}] The function~$\Bell$ satisfies the boundary condition~$\Bell(x,x^2,z) = f(x)$ 
when~$z\in [0,1]$. 
\item[\textup{(iii)}] The function~$\Bell$ satisfies the {\bf main inequality}
\begin{equation*}
\Bell(x,y,z) \geq \sum\limits_{j=1}^N\alpha_j\Bell(x_j,y_j,z_j),
\end{equation*}
whenever
\begin{equation}\label{SplittingRules}
\begin{gathered}
\sum\limits_{j=1}^N\alpha_j =1,\quad \alpha_{j} \in [0,1];
\\
\sum\limits_{j=1}^N\alpha_jx_j = x;\quad\qquad\sum\limits_{j=1}^N\alpha_jy_j = y;
\\
\forall j \quad z_j^2=z^2+(x_j-x)^2;\qquad (x_j,y_j,z_j)\in\Omega,\qquad (x,y,z)\in \Omega.
\end{gathered}
\end{equation}
\item[\textup{(iv)}] Let~$G\colon \Omega \to \mathbb{R}$ be a function that satisfies the 
same boundary condition as~$\Bell$ and also the main inequality\textup, that is
\begin{equation}\label{eq200102}
G(x,y,z) \geq \sum\limits_{j=1}^N\alpha_j G(x_j,y_j,z_j)
\end{equation}
whenever the points satisfy the splitting rules~\eqref{SplittingRules}. Then\textup,~$\Bell \leq G$ pointwise.
\end{itemize}
\end{Le}

\begin{proof}[Proof of \textup{(i)}.] Due to the assumption~$f \geq 0$ and~\eqref{eq150201}, 
the assertion that~$\Bell(x,y,z)$ is non-negative means that there exists at least one 
martingale~$\varphi$ such that
\begin{equation}\label{eq150202}
\varphi_0 = x,\quad \E\varphi_{\infty}^2 = y,\quad \text{and}
\quad (S\varphi)^2 + z^2 \leq 1\quad \text{almost surely}.
\end{equation}
We first prove that the existence of such a martingale~$\varphi$ implies~$(x,y,z) \in \Omega$. 
The necessity of~$x^2 \leq y$ follows from the Cauchy--Schwarz inequality. The necessity 
of~$y \leq 1-z^2 + x^2$ is a consequence of the~$L_2$ orthogonality:
\begin{equation}\label{TestMartingaleCondition}
y - x^2 = \E\varphi_{\infty}^2 - (\E\varphi_\infty)^2 = \E (S\varphi)^2 \leq 1-z^2.
\end{equation}

Second, for any~$(x,y,z)\in\Omega$, we may construct a single step martingale~$\varphi$ by the formula
\begin{equation*}
\varphi_0=x, \quad \varphi_1 = 
\begin{cases}
x-\sqrt{y-x^2},\quad\text{with probability}\ \frac12;
\\
x+\sqrt{y-x^2},\quad\text{with probability}\ \frac12,\end{cases}\quad \varphi_n=\varphi_1,\quad n\geq1.
\end{equation*}
Then~$\E\varphi_\infty^2=\varphi_1^2 = y$ and~$S\varphi = \sqrt{y-x^2} \leq \sqrt{1-z^2}$ almost surely.

{\it Proof of \textup{(ii)}.} If~$y=x^2$, then any martingale~$\varphi$ that satisfies~\eqref{eq150202} 
is a constant. Thus, the set of martingales over which we optimize in~\eqref{Bellman} consists of 
a single martingale that equals~$x$ identically. For such a martingale,~$\E f(\varphi_\infty) = f(x)$. 
Therefore,~$\Bell(x,x^2,z) = f(x)$, whenever~$z\in [0,1]$.

{\it Proof of \textup{(iii)}.} Let~$\eta>0$ be a small parameter to be chosen later. Pick 
some~$\alpha_j$,~$x_j$,~$y_j$, and~$z_j$,~$j \in \{1,\ldots,N\}$, that satisfy~\eqref{SplittingRules}. 
By formula~\eqref{Bellman}, for every~$j\in \{1,\ldots,N\}$, there exists a simple martingale~$\varphi^j$ 
such that
\begin{equation}\label{eq200101}
\varphi_0^j = x_j,\quad \E(\varphi^j_{\infty})^2 = y_j,\quad(S\varphi^j)^2+z_j^2\leq1\text{ almost surely},
\end{equation}
and
\begin{equation}\label{eq150203}
\Bell(x_j,y_j,z_j) \leq \E H_f\Big(\varphi^j_{\infty},(S\varphi^j)^2 +z_j^2\Big) + \eta.
\end{equation}

We split the probability space~$X$ into~$N$ parts~$X_j$ such that~$P(X_j) = \alpha_j$ (recall that 
our probability space does not have atoms). We treat each~$(X_j,\Sigma|_{X_j},\frac{1}{\alpha_j}P|_{X_j})$ 
as an individual probability space and model the martingale~$\varphi^j$ on it (this equips these 
``small'' probability spaces with some filtrations). We construct the simple martingale~$\varphi$ 
as a concatenation of these martingales:
\begin{equation*}
\varphi_0=x,\quad\forall n\in\mathbb{N}\quad\varphi_n=\sum\limits_{j=1}^N\varphi^j_{n-1}\chi_{_{X_j}}.
\end{equation*}
The constructed process $\varphi$ is a martingale because $\varphi_0= \E \varphi_1$ due 
to~\eqref{SplittingRules} and~\eqref{eq200101}. Then,~$\E\varphi_{\infty}^2 = y$ and
\eq{
(S\varphi)^2 + z^2 = (S\varphi^j)^2 + z_j^2
} 
on~$X_j$ for any $j$ by~\eqref{SplittingRules}. Therefore, 
$$
\Bell(x,y,z) \geq \E H_f\big(\varphi_{\infty},(S\varphi)^2 + z^2\big) 
=\sum\limits_{j=1}^N\alpha_j \E H_f\big(\varphi_{\infty}^j,(S\varphi^j)^2 + z_j^2\big) 
\geq \sum\limits_{j=1}^N\alpha_j\Bell(x_j,y_j,z_j) -\eta
$$
by~\eqref{eq150203}. We complete the proof by making~$\eta$ arbitrarily small.

{\it Proof of \textup{(iv)}.} If we define
\begin{equation}\label{StoppedSquareFunction}
S\varphi_n = \Big(\sum_{m < n}(\varphi_{m+1}-\varphi_m)^2\Big)^{\frac12},
\end{equation}
then by the main inequality~\eqref{eq200102} the process
\begin{equation*}
\bigg\{G\Big(\varphi_n,\E(\varphi_{\infty}^2\mid \F_n), (S\varphi_n)^2 + z^2\Big)\bigg\}_n, 
\end{equation*}
is a submartingale, which stabilizes for $n$ sufficiently large, whenever~$\varphi$ is a simple 
martingale adapted to~$\F$. Then,
\begin{multline}
G(x,y,z) = \E G(\varphi_0,\E\varphi_{\infty}^2, z^2) \geq 
\lim\limits_{n\to \infty} \E G\Big(\varphi_n,\E(\varphi_{\infty}^2\mid \F_n), (S\varphi_n)^2 + z^2\Big) = 
\\ 
\E G(\varphi_{\infty},\varphi_{\infty}^2,(S\varphi)^2 + z^2) = \E f(\varphi_{\infty}),
\end{multline} 
whenever~$\varphi$ is a simple martingale such that~$x = \varphi_0$,~$y=\E\varphi_{\infty}^2$, 
and~$(S\varphi)^2 \leq 1-z^2$. Taking supremum over all such simple martingales, 
we obtain~$G(x,y,z) \geq \Bell(x,y,z)$.
\end{proof}

\begin{Rem}\label{FinitenessRemark}
Note that \textup{(iv)} says that if there exists some function~$G$ satisfying the requirements 
of this part\textup, then~$\E f(\varphi_\infty) \leq G(\varphi_0, \E\varphi_{\infty}^2, 0)$ for 
any simple~$\varphi$ with~$S\varphi \leq 1$. 
\end{Rem}

The boundary~$y-x^2 = 1-z^2$ is somehow special for our considerations. If the 
inequality~\eqref{TestMartingaleCondition} turns into equality, then~$S\varphi = \sqrt{1-z^2}$ 
almost surely. Thus, 
\eq{\label{UpperBoundary}
\Bell(x,x^2 + s^2,\sqrt{1-s^2}) = \sup\Set{\E f(\varphi)}{\varphi_0 = x,\;S\varphi=s\;\text{ almost surely}}.
}
The extremal problem on the right hand side is interesting in itself. 

We present a simple geometric observation that Lemma~\ref{SimpleMajorant} is based upon. 
Recall the definition~\eqref{SmallBellman} of the domains $\omega_\eps$.

\begin{Le}\label{SimpleGeometry}
Let the point~$(x,y,z)\in \Omega$ be split into the points~$(x_j,y_j,z_j)$ lying inside~$\Omega$ 
according to the rules~\eqref{SplittingRules}. Then\textup, the convex hull of the points~$(x_j,y_j)$ 
lies in the parabolic strip~$\omega_{\sqrt{1-z^2}}$.
\end{Le}

\begin{proof}
We will prove that the points~$(x_j,y_j)$ lie below the tangent at~$(x,x^2 + 1-z^2)$ to the upper 
boundary of~$\omega_{\sqrt{1-z^2}}$. Note that the statement and the rules~\eqref{SplittingRules} 
are invariant with respect to the parabolic shift
\eq{\label{ParabolicShift}
(x_j,y_j,z_j)\mapsto (x_j - \tau, y_j +\tau^2 - 2\tau x_j,z_j),
} 
for any $\tau \in \mathbb{R}$. So, in what follows we may assume~$x=0$ (otherwise $x$ can be 
shifted to $0$ using the shift with $\tau = x$). For any~$j$,
\begin{equation*}
y_j \leq 1-z_j^2 + x_j^2,
\end{equation*}  
simply because~$(x_j,y_j,z_j)\in \Omega$. Therefore, by the last rule in~\eqref{SplittingRules} 
and the assumption~$x=0$,
\begin{equation*}
y_j \leq 1-z_j^2 + x_j^2 = 1-z^2,
\end{equation*}
which  means exactly that~$(x_j,y_j)$ lies below the tangent to 
the parabola~${\bf y} = {\bf x}^2 + 1-z^2$ at the point $(0,1-z^2)$. 
\end{proof}

\begin{proof}[Proof of Lemma~\textup{\ref{SimpleMajorant}}]
We have the following chain of inequalities:
\begin{equation}\label{ChainOfInequalities}
b_{\sqrt{1-z^2}}(x,y)\geq \sum\limits_{j=1}^n\alpha_jb_{\sqrt{1-z^2}}(x_{{j}},y_{{j}}) \geq 
\sum\limits_{j=1}^n\alpha_jb_{\sqrt{1-z_j^2}}(x_{{j}},y_{{j}}).
\end{equation}
The first inequality follows from the local concavity of~$b_{\sqrt{1-z^2}}$ and the fact that 
the convex hull of $(x_{{j}},y_{{j}})$ lies in $\omega_{\sqrt{1-z^2}}$ by Lemma~\ref{SimpleGeometry}. 
The second inequality is a consequence of the fact that~$b_{\eps}$ is an increasing function of~$\eps$ 
(we maximize over a larger set in~\eqref{SmallBellman} when we increase $\varepsilon$).

So, the function~$(x,y,z)\mapsto b_{\sqrt{1-z^2}}(x,y)$ satisfies the first three requirements of 
Lemma~\ref{PropertiesOfBellmanFunction}.  Thus, by the fourth statement in 
Lemma~\ref{PropertiesOfBellmanFunction} we have~$\Bell(x,y,z)\leq b_{\sqrt{1-z^2}}(x,y)$.
\end{proof}

\begin{proof}[Proof of Corollary~\textup{\ref{ExtremeSummabilityCor}}.]
By Theorem~$6.1.2$ in~\cite{ISVZ2018}, the function~$b_1$ (with~$h$ in the role of~$f$) is finite. 
Combining this information with Lemma~\ref{SimpleMajorant}, we obtain the finiteness of~$\Bell$, 
which, in the light of Remark~\ref{FinitenessRemark}, means exactly the assertion of the Corollary. 
\end{proof}

\begin{proof}[Proof of Theorem~\textup{\ref{Embedding}}]
Assume that~$\|S\varphi\|_{L^{\infty}}= 1$. Let us show that in this case the~$\mathbb{R}^2$-valued 
martingale $M_n = (\varphi_n, \E(\varphi^2\mid \F_n))$ is an~$\omega_1$-martingale. We verify three 
conditions in Definition~\ref{BasicMartingales}. The second condition is justified by the martingale 
convergence theorem since~$\varphi \in L_2$. To verify the third property, we consider 
an~$\mathbb{R}^3$-valued process
\eq{
\mu_n = (\varphi_n, \E(\varphi^2\mid \F_n), S\varphi_{n}),
} 
where~$S\varphi_n$ is defined in~\eqref{StoppedSquareFunction}.
Let~$a\in\mathcal{F}_n$ be an atom. Then, the points~$(x,y,z)=\mu_n(a)$ and~$(x_j,y_j,z_j)=\mu_{n+1}(a_j)$, 
where the~$a_j$ are all the children of~$a$, satisfy~\eqref{SplittingRules}. Thus, by 
Lemma~\ref{SimpleGeometry}, the convex hull of the points~$M_{n+1}(a_j)$ lies inside~$\omega_1$. 
Therefore,~$M$ is an~$\omega_1$-martingale.

Recall~$M_{\infty}^1$ is the first coordinate of~$M_{\infty}$. By Lemma~\ref{MartFunct},
$\|(M_{\infty}^1)^*\|_{\BMO([0,1])} \leq 1$ since~$M$ is an $\omega_1$-martingale. We notice 
that~$M_{\infty}^1$ coincides with~$\varphi_{\infty}$ and finally obtain the inequality
\begin{equation*}
\|\varphi_{\infty}^*\|_{\BMO([0,1])} \leq \|S\varphi\|_{L_{\infty}}.
\end{equation*}

The sharpness of this inequality is obtained by considering the martingale~$\varphi$ such that
$\varphi_0 = 0$ and~$\varphi_1$ is~$\pm1$ with equal probability.
\end{proof}

The lemma below suggests a simpler way to verify property~\textup{(iii)} of 
Lemma~\ref{PropertiesOfBellmanFunction}.

\begin{Le}\label{main_under_tangent_Le} 
Let~$G\colon \Omega \to \mathbb{R}$ be a function. Assume that for every point 
$(x,y,z)\in\Omega\setminus\{y=x^2\}$ there exist numbers~$\ell_1(x,y,z)$ and~$\ell_2(x,y,z)$ 
such that the estimate
\begin{equation}\label{under_tangent_relation_Le}
    G(\xx, \yy, \zz) \le G(x,y, z) + \ell_1(x,y,z)(\xx - x) + \ell_2(x,y,z) (\yy - y)
\end{equation}
holds true for every point~$(\xx, \yy, \zz) \in \Omega$ such that~$\zz^2 = z^2+ (\xx-x)^2$. 
Then\textup,~$G$ fulfills the main inequality
\begin{equation}
G(x,y,z) \geq \sum\limits_{j=1}^N \alpha_j G(x_j,y_j,z_j),
\end{equation}
where the parameters involved satisfy the splitting rules~\eqref{SplittingRules}.
\end{Le}

\begin{Rem}
If~$G$ is differentiable at~$(x,y,z)$\textup, then the natural choice for~$\ell_1(x,y,z)$ and 
$\ell_2(x,y,z)$ would be the pair of partial derivatives~$\frac{\partial }{\partial x}G(x,y,z)$ 
and~$\frac{\partial }{\partial y}G(x,y,z)$. In fact\textup, one may show a reverse statement. 
If~$G$ satisfies the main inequality as above and is~$C^1$-smooth on~$\Omega,$ 
then~\eqref{under_tangent_relation_Le} is true with~$\ell_1$ and~$\ell_2$ being the corresponding 
partial derivatives of~$G$ at~$(x,y,z)$.
\end{Rem}

\begin{proof}[Proof of Lemma~\textup{\ref{main_under_tangent_Le}}.]
Pick some collection of parameters that satisfy the splitting rules~\eqref{SplittingRules}. 
Without loss of generality, we may assume~$y > x^2$ (in this case the main inequality is 
trivial since in this case $y_j=y$, $x_j=x$). Setting $(\xx,\yy) = (x_j,y_j)$, we obtain
\begin{equation}\label{B_under_tangent_j}
G(x_j, y_j,z_j) \le G(x,y,z) +  \ell_1(x,y,z)(x_j - x)+ \ell_2(x,y,z) (y_j - y),
\end{equation} 
for every $j\in \{1,\ldots,N\}$.
Multiplying~\eqref{B_under_tangent_j} by $\alpha_j$ and summing these products, we obtain 
the desired inequality
\begin{multline}
\sum\limits_{j=1}^N\alpha_jG(x_j,y_j,z_j) \le
\\
G(x,y,z) + \left( \sum\limits_{j=1}^N \alpha_j x_j - x \right)  \ell_1(x,y,z)  + 
\left( \sum\limits_{j=1}^N \alpha_j y_j - y \right)\ell_2(x,y,z) = G(x,y,z).
\end{multline}
\end{proof}

\section{Simple cases}\label{S3}
\subsection{Foliations for Bellman functions}\label{s31}
We describe the function~$b_\eps$ defined in~\eqref{SmallBellman} in the cases needed for the proof 
of Theorem~\ref{CoincidenceTheorem}. We refer the reader to~\cite{ISVZ2018} for details; as it has 
been said, some of these results were obtained in earlier papers. 

Consider the case~$f''$ is non-increasing on the entire line (recall $f$ is twice differentiable). 
For any $u \in \R$, we draw the segment
\eq{
\Big[(u,u^2), \big((u-\eps),(u-\eps)^2 + \eps^2\big)\Big]
}
that touches the upper boundary of~$\omega_\eps$. Note that when~$u$ runs through~$\R$ these 
segments foliate the entire domain~$\omega_\eps$. We call such segments right tangents (since 
they lie on the right of the tangency point). For any~$(x,y)\in \omega_\eps$ there is a unique right 
tangent that passes through it. We denote the corresponding point~$u$ by~$\uR(x,y)$. In other words,
\eq{
(x,y) \in \Big[(\uR,\uR^2), \big((\uR-\eps),(\uR-\eps)^2 + \eps^2\big)\Big].
}
\begin{Th}\label{ThRTang}
Let~$f$ satisfy the standard requirements \textup(Definition~\ref{DefStReq}\textup) and let~$f''$ 
be non-increasing. The function~$b_{\eps}$ is linear along right tangents in the sense that there 
exists a function~$\mr\colon \R \to \R$ such that
\eq{\label{RightLinearity}
b_{\eps}(x,y) = f(\uR) + \mr(\uR)(x-\uR),\quad \uR = \uR(x,y),\quad (x,y)\in \omega_\eps.
}
The value of~$\mr$ may be computed by the formula
\eq{
\label{eq270102}
\mr(u) = \eps^{-1}f(u) - \eps^{-2}\int\limits_{-\infty}^0 e^{t/\eps} f(u+t)\,dt.
}
\end{Th}

The case when~$f''$ is non-decreasing is completely similar. In this case, we consider left tangents
\eq{
\Big[(u,u^2), \big((u+\eps),(u+\eps)^2 + \eps^2\big)\Big]
}
and the corresponding function~$\uL\colon \omega_\eps \to \R$ such that
\eq{\label{DefUl}
(x,y) \in \Big[(\uL,\uL^2), \big((\uL+\eps),(\uL+\eps)^2 + \eps^2\big)\Big].
}
\begin{Th}\label{TheoremLeftTangents}
Let~$f$ satisfy the standard requirements \textup(Definition~\ref{DefStReq}\textup) and 
let~$f''$ be non-decreasing. The function~$b_{\eps}$ is linear along left tangents in the sense 
that there exists a function~$\ml\colon \R \to \R$ 
such that
\eq{\label{LeftLinearity}
b_{\eps}(x,y) = f(\uL) + \ml(\uL)(x-\uL),\quad \uL = \uL(x,y),\quad (x,y)\in \omega_\eps.
}
The value of~$\ml$ may be computed by the formula
\eq{\label{eq270104}
\ml(u) = -\eps^{-1}f(u) + \eps^{-2}\int\limits_0^{\infty} e^{-t/\eps} f(u+t)\,dt.
}
\end{Th}

Now consider the case where there exists a point~$c\in \R$ such that~$f''$ is non-decreasing 
on the left of~$c$ and is non-increasing on the right. In this case, there exist unique continuous 
functions~$a,b \colon [0,2\eps] \to \mathbb{R}$ such that~$a$ is decreasing,~$b$ is increasing, and
\begin{gather}
a(0) = b(0) = c;
\\
b(l) - a(l) = l,\qquad l \in [0,2\eps];\rule{0pt}{12pt}
\\
\label{UrLun}
\frac{f'(b)+f'(a)}2=\frac{f(b)-f(a)}{b-a},\qquad a=a(l),\; b= b(l),\quad l\in (0,2\eps].
\end{gather}
We split~$\omega_\eps$ into three domains
\alg{
\label{om1}\om_1(\eps) &= \Set{(x,y)\in \omega_\eps}{\uL(x,y) \leq a(2\eps)};
\\
\label{om2}\om_2(\eps) &= \Set{(x,y)\in\omega_\eps}{\uL(x,y) \geq a(2\eps), \ \uR(x,y) \leq b(2\eps)};
\\
\label{om3}\om_3(\eps) &= \Set{(x,y)\in\omega_\eps}{\uR(x,y) \geq b(2\eps)},
}
the first and third of them called tangent domains, the second called a cup. 
The identity~\eqref{UrLun} is called the cup equation.
\begin{Th}\label{LeftToRight}
Let~$f$ satisfy the standard requirements (Definition~\textup{\ref{DefStReq})}. 
Assume~$f''$ be non-decreasing on the left of~$c$ and non-increasing on the right. 
The function~$b_\eps$ is linear along the chords
\eq{\label{Chord}
\Big[\big(a(l),a^2(l)\big), \big(b(l),b^2(l)\big)\Big],\quad l\in (0,2\eps],
}
in the sense that
\eq{\label{InCup}
\begin{gathered}
b_\eps(x,y) = \alpha f(a(l)) + \beta f(b(l)), \qquad \text{whenever} 
\\ 
x=\alpha a(l) + \beta b(l),\quad y = \rule{0pt}{15pt}
\alpha a^2(l)+\beta b^2(l),\qquad \alpha + \beta = 1,\quad \alpha,\beta \geq 0.
\end{gathered}
}
This defines the function~$b_\eps$ in the cup~\eqref{om2} foliated by the chords. 
On the tangent domains~\eqref{om1} and~\eqref{om3}\textup, the function~$b_\eps$ is defined 
by formulas~\eqref{LeftLinearity} and~\eqref{RightLinearity} respectively. 
The corresponding functions~$\ml$ and~$\mr$ are given by the formulas
\alg{\label{mlcup}
\ml(u) &=\frac{f(b(2\eps))\!+\!f(a(2\eps))}{2\eps}\exp\Big(\frac{u\!-\!a(2\eps)}{\eps}\Big)-\frac{f(u)}{\eps}+  
\frac{1}{\eps^{2}}\!\!\!\!\!\int\limits_{0}^{a(2\eps)-u}\!\!\!\!\! e^{\!-t/\eps}f(t+u)\,dt,\!\!
\quad u \in (\!-\infty, a(2\eps));
\\
\label{mrcup}
\mr(u) &=\frac{f(b(2\eps))\!+\!f(a(2\eps))}{2\eps}\exp\Big(\frac{b(2\eps)\!-\!u}{\eps}\Big)+\frac{f(u)}{\eps}-  
\frac{1}{\eps^{2}}\!\!\!\!\int\limits_{b(2\eps)-u}^{0}\!\!\!\! e^{t/\eps}f(t+u)\,dt,
\quad u \in (b(2\eps),+\infty).
}
\end{Th}

\subsection{Useful lemmas}\label{s32}

\begin{Le}\label{AlongParabola}
Let~$(x_0,y_0,z_0)\in\Omega,$ $y-x^2 = y_0-x_0^2>0,$ and~$x > x_0$. Let~$G\colon\Omega\to\R$ be a function that 
satisfies the first three properties in Lemma~\textup{\ref{PropertiesOfBellmanFunction}} with a function $f$ 
continuous on $[x_0-\eps,x-\eps],$ where $\eps = \sqrt{y_0 - x_0^2}$. Then\textup, we have
\eq{\label{IntegralEstimate}
G(x_0,y_0,z_0) \geq \eps^{-1}\int\limits_{x_0}^x e^{-\frac{\tau-x_0}{\eps}} f(\tau-\eps)\,d\tau + 
e^{-\frac{x-x_0}{\eps}}\liminf_{\delta\to 0+} G\Big(x,y-\delta,\sqrt{z_0^2 + \delta}\Big).
}
\end{Le}
\begin{proof}
Let~$N$ be a large number, let~$t = \frac{x-x_0}{N}$. We construct the points~$(x_n,y_n,z_n)$,
$n\in\{1,\ldots,N\}$, consecutively, starting from~$(x_0,y_0,z_0)$:
\alg{
x_{n+1} &= x_n +t;
\\
y_{n+1} &= x_{n+1}^2 + y_n - x_n^2 - t^2;
\\
z_{n+1}^2 &= z_n^2 + t^2.
}
We note that 
$$y_{n+1}-x_{n+1}^2 = y_n-x_n^2 -t^2\leq 1-z_n^2 - t^2 = 1-z_{n+1}^2$$ and 
$$y_n-x_n^2 \geq y_N-x_N^2 = y_0-x_0^2 - Nt^2 = y_0 - x_0^2 - \frac{(x-x_0)^2}{N}>0$$
for $N$ large enough. Thus, all the points $(x_n,y_n,z_n)$ belong to $\Omega$. It is also convenient 
to introduce a sequence of parameters~$\eps_n$, where~$\eps_n^2 = y_n - x_n^2$. Then,
\eq{
\eps_{n}^2 = \eps^2 - nt^2.
}
The point~$(x_n,y_n,z_n)$ splits into~$(x_{n+1},y_{n+1},z_{n+1})$ and 
$(x_n-\eps_n,(x_n-\eps_n)^2,\sqrt{z_n^2+\eps_n^2})$ 
according to the rules~\eqref{SplittingRules}, which allows to write
\eq{
G(x_n,y_n,z_n) \geq \frac{\eps_n}{t+ \eps_n} G(x_{n+1},y_{n+1},z_{n+1}) + \frac{t}{t+\eps_n} f(x_n -\eps_n).
}

If we combine these inequalities, we arrive at
\eq{\label{PreIntegral}
\begin{aligned}
G(x_0,y_0,z_0)&\geq\Big(\prod_{n=0}^{N-1}\frac{\eps_n}{\eps_n+t}\Big)G\Big(x,y-Nt^2,\sqrt{z_0^2+Nt^2}\Big) 
\\
&\qquad\qquad+\sum\limits_{n=0}^{N-1}\frac{t}{t+\eps_n}\Big(\prod_{j=0}^{n-1}\frac{\eps_j}{\eps_j + t}\Big) 
f(x_0+tn-\eps_n).
\end{aligned}
}
It remains to prove that the sum on the right hand side converges as~$N\to \infty$ to the right hand side 
of~\eqref{IntegralEstimate}. This is, in fact, a fairly lengthy calculus exercise. We comment on its proof 
without going deeply into details. The main ``engine'' of this effect is that  we have 
$\eps_j = \eps + O(t)$,~$z_j = z_0+ O(t)$ uniformly in~$j\in \{0,\ldots,N\}$ when~$N$ is large. 
This allows to write
\eq{
\prod_{j=0}^{n-1}\frac{\eps_j}{\eps_j + t} = e^{-\frac{nt}{\eps}} + O(t)
}
uniformly in~$n\in \{1,\ldots,N\}$. Recalling~$Nt = x-x_0$, we get 
\eq{\label{eq270101}
\liminf_{N\to\infty}\Big(\prod_{n=0}^{N-1}\frac{\eps_n}{\eps_n+t}\Big)G\Big(x,y-Nt^2,\sqrt{z_0^2+Nt^2}\Big) 
\geq e^{-\frac{x-x_0}{\eps}}\liminf_{\delta\to 0+} G\Big(x,y-\delta,\sqrt{z_0^2 + \delta}\Big).
}
The second term in~\eqref{PreIntegral} equals
\eq{
\begin{aligned}
\sum\limits_{n=0}^{N-1}\frac{t}{t+\eps_n}\Big(\prod_{j=0}^{n-1}\frac{\eps_j}{\eps_j + t}\Big) f(x_0+tn-\eps_n) 
&= \frac{t}{\eps} \sum\limits_{n=0}^{N-1} e^{-\frac{nt}{\eps}}f(x_0+tn-\eps_n) +  O(t) 
\\
&=\eps^{-1}\int\limits_{x_0}^x e^{-\frac{\tau-x_0}{\eps}} f(\tau-\eps)\,d\tau + o(1).
\end{aligned}
}
\end{proof}

\begin{Rem}
In the case~$x_0 > x,$ the estimate~\eqref{IntegralEstimate} should be replaced with
\eq{
G(x_0,y_0,z_0) \geq \eps^{-1}\int\limits_{x}^{x_0} e^{\frac{\tau-x_0}{\eps}} f(\tau+\eps)\,d\tau + 
e^{\frac{x-x_0}{\eps}}\liminf_{\delta\to 0+} G\big(x,y-\delta,\sqrt{z_0^2 + \delta}\big).
}
\end{Rem}

\begin{Rem}\label{Rem270101}
Inequalities~\eqref{PreIntegral} and~\eqref{eq270101} lead to the following assertion. 
Let~$(x_0,y_0,z_0)\in\Omega,$ $y-x^2 = y_0-x_0^2>0,$ and~$x > x_0$. Let~$G\colon\Omega\to\R$ be 
a function that satisfies the first three properties in Lemma~\textup{\ref{PropertiesOfBellmanFunction}} 
with a function $f$ non-negative on $[x_0-\eps,x-\eps],$ where $\eps=\sqrt{y_0 - x_0^2}$. Then we have
\eq{
G(x_0,y_0,z_0) \geq e^{-\frac{x-x_0}{\eps}}\liminf_{\delta\to 0+} G\Big(x,y-\delta,\sqrt{z_0^2 + \delta}\Big).
}
Here we require no continuity assumption on~$f$.
\end{Rem}

\begin{Le}
Let~$G\colon \Omega \to \R$ be a function that satisfies the first three properties in 
Lemma~\textup{\ref{PropertiesOfBellmanFunction}}. Fix some~$z_0\in(0,1)$. Let~$(x,y)\in\R^2$ be a point such that
\alg{
\label{InDomain}0\leq y-x^2 &< 1-z_0^2;
\\
\label{CurvedParabola}y \geq &\; 2x^2.
}
Let also~$(x_0,y_0) = \alpha (x,y),$ where~$\alpha \in (0,1)$. Then\textup, 
\eq{\label{DesiredConvexity}
G(x_0,y_0,z_0) \geq \alpha \liminf_{z\to z_0+} G(x,y,z) + (1-\alpha) f(0).
}
\end{Le}
\begin{proof}
Let~$\lambda=\alpha^{-1}$, $A=(x,y,z_0)$, and~$A_0=(x_0,y_0,z_0)$. In particular, $(x,y)=\lambda(x_0,y_0)$. 
Let~$N$ be a large number to be specified later. Consider the points $A_n=(x_n,y_n,z_n)$, $n\in\{0,\ldots,N\}$, 
defined consecutively
\eq{
\begin{gathered}
(x_n,y_n) = \lambda^{\frac{n}{N}}(x_0,y_0);\qquad
\label{zndef}z_n^2 = z_{n-1}^2 + (x_n-x_{n-1})^2.
\end{gathered}
}
In other words, the point~$A_{n-1}$ splits into~$A_n$ and~$(0,0,t_n)$, where
\eq{\label{DefinitionOftn}
t_n^2 = z_{n-1}^2 + x_{n-1}^2,
}
according to the rules~\eqref{SplittingRules} (provided we assume $A_n\in\Omega$; we will approve 
this assumption slightly later). We may provide an explicit formula for~$z_n$:
\eq{
z_n^2 = z_0^2 + \sum\limits_{k=0}^{n-1}\lambda^{\frac{2k}N}(\lambda^{\frac1N}-1)^2x_0^2 = 
z_0^2+(\lambda^{\frac1N} - 1)\frac{\lambda^{\frac{2n}N} - 1}{\lambda^{\frac1N}+1}x_0^2.
}
In particular,~$z_N\to z_0+$ when~$N\to \infty$. Therefore,~$A_N \to A$. Since we have assumed strict 
inequality in~\eqref{InDomain}, we have $A_N\in\Omega$ provided $N$ is sufficiently large. 

Since the constructed points satisfy the splitting rules~\eqref{SplittingRules} and
\begin{equation}
    (x_n,y_n) = \lambda^{\frac{1}{N}}(x_{n-1},y_{n-1}),
\end{equation}
we may write the inequalities
\eq{\label{eq337}
G(A_{n-1}) \geq \alpha_N G(A_n) + (1-\alpha_N)f(0),\quad \alpha_{N} = \lambda^{{-}\frac{1}{N}},
}
provided we verify that the points~$A_n$ and~$(0,0,t_n)$ belong to~$\Omega$ for any~$n$. 
We multiply~\eqref{eq337} by~$\alpha_N^{n-1}$, sum over all~$n$, and obtain
\eq{
G(A_0) \geq \alpha G(A_N) + (1-\alpha) f(0),
}
which implies~\eqref{DesiredConvexity} since~$z_N\to z_0+$ when~$N\to \infty$. 

It remains to verify the inequalities $y_n-x_n^2\leq 1-z_n^2$ and $t_{n-1}^2\leq1$ for any $n\in\{0,\ldots,N\}$. 
Note that the quantity~$1-z_n^2$ is a non-increasing function of~$n$ (by~\eqref{zndef}), whereas~$y_n - x_n^2$ 
is non-decreasing (by~\eqref{CurvedParabola} and~\eqref{zndef}). Thus, the inequality
$y_n-x_n^2\leq 1-z_n^2$ for smaller~$n$ is a consequence of the same inequality with~$n=N$; 
the latter inequality follows from~$A_N \in \Omega$. 

The same principle allows to establish the second inequality since~$t_n$ defined in~\eqref{DefinitionOftn}
is an increasing function of~$n$. Thus, it suffices to verify~$t_{N-1} \leq 1$, which is a 
consequence of~$z_N \to z_0+$ and~$z_0^2 + x^2 < 1$. The latter inequality follows from~\eqref{InDomain} 
and~\eqref{CurvedParabola}.
\end{proof}

\begin{Rem}\label{PashaArbitraryPoint}
We may replace the point~$(0,0)$ with an arbitrary point~$(t,t^2)$ with the help of a parabolic 
shift~\eqref{ParabolicShift} in the lemma above. Here the resulting statement is\textup, with the same 
function~$G$. Let~$(x,y) \in \R^2$ be a point such that
\begin{gather}
0\leq y-x^2 < 1-z_0^2;
\\
y - t^2 \geq 2(x-t)^2.
\end{gather} 
Let also~$(x_0,y_0) = \alpha (x,y) + (1-\alpha)(t,t^2),$ where~$\alpha \in (0,1)$. Then\textup, 
\eq{
G(x_0,y_0,z_0) \geq \alpha \liminf_{z\to z_0+}G(x,y,z) + (1-\alpha) f(t).
}
\end{Rem}
\begin{Rem}
The proof may be modified to obtain a priori stronger inequality
\eq{\label{DesiredConvexity_mod}
G(x_0,y_0,z_0) \geq \alpha \limsup_{z\to z_0 +}G(x,y,z) + (1-\alpha) f(t).
}
\end{Rem}

\subsection{Proof of Theorem~\ref{CoincidenceTheorem}}\label{s330}

\begin{Th}\label{ThOnChord}
Let~$f$ be continuous at $a,b \in \mathbb{R}$. Assume that the function $b_{\sqrt{1-z_0^2}}$ is linear along the 
segment~$\ell=\big[(a,a^2),(b,b^2)\big]\subset \omega_{\sqrt{1-z_0^2}}$. 
Then~$\Bell(x_0,y_0,z_0)=b_{\sqrt{1-z_0^2}}(x_0,y_0)$ whenever~$(x_0,y_0) \in \ell$. 
\end{Th}

\begin{proof}
Let~$q$ be the midpoint of~$\ell$. Then,
\eq{
\Bell(q,z_0) \geq \frac{f(a) + f(b)}{2} =b_{\sqrt{1-z_0^2}}(q),
}
since~$(q,z_0)$ might be split into the points
\eq{
\bigg(a,a^2,\sqrt{z_0^2+\tfrac{(b-a)^2}4}\bigg)\quad\text{and}\quad\bigg(b,b^2,\sqrt{z_0^2+\tfrac{(b-a)^2}4}\bigg)
}
according to the rules~\eqref{SplittingRules} (note that the said points lie in~$\Omega$). 
Thus, by Lemma~\ref{SimpleMajorant}, we have 
\eq{
\Bell(q,z_0) = b_{\sqrt{1-z_0^2}}(q).
} 
Let now~$(x_0,y_0)$ lie on~$\ell$ on the left of~$q$. Remark~\ref{PashaArbitraryPoint} implies
\eq{
\Bell(x_0,y_0,z_0)\geq\alpha\liminf_{z\to z_0+}\Bell(x,y,z)+(1-\alpha)f(a),\qquad\alpha=\frac{x_0-a}{x-a},
}
for any point~$(x,y) \in \ell$ lying arbitrarily close to~$q$. Similar to the reasoning for 
the point~$(q,z_0)$ above, 
\eq{\label{popolam}
\Bell(x,y,z) \geq \frac{f(x-\sqrt{y-x^2}) + f(x+ \sqrt{y-x^2})}{2},
}
which implies (with the same notation~$\alpha = \frac{x_0 - a}{x-a}$)
\eq{
\Bell(x_0,y_0,z_0) \geq  \liminf_{(x,y) \to q} \Big(\alpha\frac{f(x-\sqrt{y-x^2}) + 
f(x+ \sqrt{y-x^2})}{2} + (1-\alpha) f(a)\Big) = b_{\sqrt{1-z_0^2}}(x_0,y_0).
}
\end{proof}

\begin{Th}\label{ThTanL}
Suppose that $f$ is continuous and non-negative\textup, $b_\eps(x,y)$ is continuous as a function 
of~$(x,y,\eps)$ on $\{x^2\leq y\leq x^2+\eps^2, 0<\eps\leq1\}$. Assume that~\eqref{LeftLinearity} 
and~\eqref{eq270104} hold true for any $\eps\in(0,1]$. 
Then\textup,~$\Bell(x_0,y_0,z_0) = b_{\sqrt{1-z_0^2}}(x_0,y_0)$ for all~$(x_0,y_0,z_0)\in\Omega$.
\end{Th}

\begin{proof}
Let us first consider the case where~$y_0 - x_0^2 = 1-z_0^2$. We apply Lemma~\ref{AlongParabola}, 
drop the second summand (using the positivity of~$\Bell$), and set~$x=\infty$:
\eq{\label{eq270103}
\Bell(x_0,y_0,z_0)\geq\eps^{-1}\!\!\int\limits_{x_0}^\infty \!e^{-\frac{\tau-x_0}{\eps}}f(\tau-\eps)\,d\tau,
\qquad \eps^2 = y_0 - x_0^2 = 1-z_0^2.
}

By our assumptions, the right hand side of~\eqref{eq270103} coincides with $b_{\eps}(x_0,y_0)$, therefore 
\eq{
\Bell(x_0,y_0,z_0) \geq b_{\eps}(x_0,y_0),\quad \eps^2 = y_0 - x_0^2 = 1-z_0^2.
}
By Lemma~\ref{SimpleMajorant}, this inequality is, in fact, an equality.

Consider now the case $y_0-x_0^2<1-z_0^2$. We split~$(x_0,y_0)$ into the convex combination 
of~$(\uL,\uL^2)$ and
\eq{
P = \bigg(\uL+\sqrt{1-z_0^2}, \Big(\uL + \sqrt{1-z_0^{2}}\Big)^2 + 1-z_0^2\bigg),
} 
along the left tangent~$\ell$ to the parabola~$y = x^2 +1-z_0^2$ at~$P$; here
\begin{equation}
\uL = \uL(x_0,y_0){=x_0-\sqrt{1-z_0^2}+\sqrt{1-z_0^2+x_0^2-y_0}}
\end{equation}
is defined in~\eqref{DefUl}.
Remark~\ref{PashaArbitraryPoint} implies
\eq{\label{ConvAlongTang}
\Bell(x_0,y_0,z_0) \geq \alpha \liminf_{z\to z_0+} \Bell(x,y,z) + (1-\alpha) f(\uL), 
\qquad \alpha = \frac{x_0 - \uL}{x-\uL},
}
for any point~$(x,y) \in \ell$ lying arbitrarily close to~$P$. Since the function $\Bell(x,y,\cdot)$ 
is non-increasing (see Remark~\ref{Monotonicity}), we have
\eq{
\Bell(x,y,z) \geq \Bell(x,y,\sqrt{1-y+x^2}) = b_{{\sqrt{y-x^2}}}(x,y),
}
the equality holds by the already considered case. We plug this back into~\eqref{ConvAlongTang}:
\eq{
\Bell(x_0,y_0,z_0) \geq \alpha  b_{{\sqrt{y-x^2}}}(x,y) + (1-\alpha) f(\uL).
}
It remains to note that when~$(x,y) \to P$, the right hand side tends to~$b_{\sqrt{1-z_0^2}}(x_0,y_0)$  
since $b$ is continuous and~$b_{\sqrt{1-z_0^2}}$ is linear along~$\ell$ by our assumptions.
\end{proof}

\begin{Th}\label{ThRTangEps}
Suppose that $f$ is continuous and non-negative\textup, $b_\eps(x,y)$ is continuous as a function 
of~$(x,y,\eps)$ on $\{x^2 \leq y \leq x^2+\eps^2, 0 < \eps \leq 1\}$. Assume that~\eqref{RightLinearity} 
and~\eqref{eq270102} hold true for any $\eps \in (0,1]$. 
Then\textup,~$\Bell(x,y,z) = b_{\sqrt{1-z^2}}(x,y)$ for all~$(x,y,z)\in\Omega$.
\end{Th}

\begin{Th}\label{CupTheorem}
Suppose that $f$ is continuous and non-negative\textup, $b_\eps(x,y)$ is continuous as a function 
of~$(x,y,\eps)$ on $\{x^2 \leq y \leq x^2+\eps^2, 0 < \eps \leq 1\}$. Assume that for any~$\eps \in (0,1]$ 
the function~$b_{\eps}$ has the following structure\textup: there exist some functions~$a$ and $b$ 
that satisfy the properties listed in Theorem~\ref{LeftToRight}\textup, and on the domain~\eqref{om2} 
formula~\eqref{InCup} holds true\textup; formulas~\eqref{LeftLinearity} and~\eqref{RightLinearity} 
\textup(with the coefficients given in~\eqref{mrcup} and~\eqref{mlcup}\textup) define~$b_\eps$ on 
the domains $\om_{1}(\eps)$ and $\om_{3}(\eps)$ given in~\eqref{om1} and~\eqref{om3} respectively. 
Then\textup,
\begin{equation}
\Bell(x_0,y_0,z_0) = b_{\sqrt{1-z_0^2}}(x_0,y_0),\quad (x_0,y_0,z_0) \in \Omega.
\end{equation}
\end{Th}

\begin{proof}
Fix~$z_0$, $\eps  = \sqrt{1-z_0^2}$, and consider the function~$b_{\eps}$ on~$\omega_{\eps}$. 
By Lemma~\ref{SimpleMajorant} we only need to prove
\eq{\label{eq270105}
\Bell(x_0,y_0,z_0) \geq b_\eps(x_0,y_0).
}

If~$(x_0,y_0)\in \om_2(\eps)$ (see~\eqref{om2}), then~\eqref{eq270105} follows from Theorem~\ref{ThOnChord}. 
In particular, for
\eq{\label{eq351}
Q = \Big(\frac{a(2\eps) + b(2\eps)}{2}, \frac{a^2(2\eps) + b^2(2\eps)}{2}\Big)
}
we have 
\eq{
\Bell(Q,z_0) \geq \frac{f(a(2\eps)) + f(b(2\eps))}{2} = b_{\sqrt{1-z_0^2}}(Q).
}

Now we deal with other points that satisfy~$y-x^2=\eps^2$. Let~$(x_0,y_0)$ with $y_0-x_0^2=\eps^2$ lie on 
the left of~$Q$ (the other case is completely similar). We apply Lemma~\ref{AlongParabola} with~$(x,y) = Q$:
\mlt{\label{TwoPointEstimateFromBelow}
\Bell(x_0,y_0,z_0) \geq \eps^{-1}\!\int\limits_{x_0}^x e^{-\frac{\tau-x_0}{\eps}} f(\tau-\eps)\,d\tau + 
e^{-\frac{x-x_0}{\eps}}\liminf_{\delta\to 0+} \Bell\big(x,y-\delta,\sqrt{z_0^2 + \delta}\big) \Geqref{popolam}
\\
\eps^{-1}\!\int\limits_{x_0}^x e^{-\frac{\tau-x_0}{\eps}} f(\tau-\eps)\,d\tau + 
e^{-\frac{x-x_0}{\eps}}\liminf_{\delta\to 0+} \frac{f(x-\sqrt{\eps^2-\delta}) + f(x+\sqrt{\eps^2-\delta})}{2} = 
\\
\eps^{-1}\!\int\limits_{x_0}^x e^{-\frac{\tau-x_0}{\eps}} f(\tau-\eps)\,d\tau + 
e^{-\frac{x-x_0}{\eps}}\frac{f(x-\eps) + f(x+\eps)}{2}.
}
A direct computation shows that the right hand side coincides with~$b_{\sqrt{1-z_0^2}}(x_0,y_0)$ described 
by Theorem~\ref{LeftToRight}. Thus, we have proved~\eqref{eq270105} for the points satisfying $y_0-x_0^2=\eps^2$.

If~$(x_0, y_0)$ lies inside~$\om_1(\eps)$ or~$\om_3(\eps)$ (see~\eqref{om1} and~\eqref{om3}), 
then~\eqref{eq270105} is proved by the same method as we used to prove Theorem~\ref{ThTanL}. 
\end{proof}

\begin{proof}[Proof of Theorem~\textup{\ref{CoincidenceTheorem}}]
If $f''$ is non-decreasing, the theorem follows from Theorems~\ref{TheoremLeftTangents} and~\ref{ThTanL}. 
If $f''$ is non-increasing, it follows from Theorems~\ref{ThRTang} and~\ref{ThRTangEps}. In the last 
case, when $f''$ changes its monotonicity, we rely upon Theorems~\ref{LeftToRight} and~\ref{CupTheorem}.
\end{proof}

\begin{proof}[Proof of Corollary~\textup{\ref{p-momentCorollary}}]
By the very definition,
\begin{equation}\label{eq160201}
c_p = \sup\limits_{0\leq y \leq 1-z^2} \Bell^{\frac1p}(0,y,z),
\end{equation}
where the function~$\Bell$ is constructed from the boundary condition~$f(t) = |t|^p$. 
Despite the fact that $f$ does not fulfill the standard requirements, the corresponding Bellman functions 
$b_\eps$ are described by the same formulas as in Theorem~\ref{LeftToRight} (see~\cite{SlavinVasyunin2012}). 
Therefore, by Theorem~\ref{CupTheorem} and Remark~\ref{Monotonicity}, the supremum in~\eqref{eq160201} 
coincides with 
\begin{equation}
\sup\limits_{0\leq y \leq 1} \Bell^{\frac1p}(0,y,0) = \sup\limits_{0\leq y\leq 1} b_{1}^{\frac1p}(0,y).
\end{equation}
The latter supremum equals~$1$ since~$b_1(0,y) = y^\frac{p}{2}$ in this case.
\end{proof}
\begin{proof}[Proof of Corollary~\textup{\ref{exponentialmomentCorollary}}]
We consider the function~$\Bell$ constructed for~$f(t) = e^{\eps t}$ and observe that
\begin{equation}
C(\eps) = \sup\limits_{0\leq y\leq 1}\Bell(0,y,0).
\end{equation}
This case falls under the scope of Theorem~\ref{CoincidenceTheorem}. Similar to the previous proof,
\begin{equation}
C(\eps) = \sup\limits_{0\leq y\leq 1} b_{1}(0,y) = \frac{e^{-\eps}}{1-\eps},
\end{equation}
as it may be derived from the exact formula for the latter function (see either 
Theorem~\ref{TheoremLeftTangents}, or the original paper~\cite{SlavinVasyunin2011}). 
\end{proof}
Finally, we present a local version of Theorem~\ref{CupTheorem}, which may be obtained by the same proof.

\begin{Th}\label{BigTheorem}
Let $0<\eps_1<\eps_2 \leq 1$. Suppose that there are continuous functions 
$a,\,b\colon[2\eps_1,2\eps_2]\to\mathbb{R}$ 
such that $a$ is decreasing\textup, $b$ is increasing\textup, and $b(l)-a(l)=l$ for $l\in[2\eps_1,2\eps_2]$. 
Let $\tl,\,\tr\in\mathbb{R}$ satisfy inequalities $\tl<a(2\eps_2),$ $b(2\eps_2)<\tr$. Suppose that $f$ is 
continuous on $[\tl, a(2\eps_1)]\cup [b(2\eps_1),\tr]$. Assume that for any~$\eps\in[\eps_1,\eps_2]$ 
the function~$b_{\eps}$ satisfies the following properties\textup: 
\begin{itemize}
	\item formula~\eqref{LeftLinearity} with the coefficients given in~\eqref{mlcup} holds on the domain
$$
\om_1(\eps;\tl) = \Set{(x,y)\in \omega_\eps}{\tl\leq \uL(x,y) \leq a(2\eps)};
$$
	\item formula~\eqref{InCup} holds on any chord $\big[(a(l),a^2(l)),(b(l),b^2(l))\big],$ $l\in[2\eps_1,2\eps],$ 
these chords foliate a domain we denote by~$\om_2(\eps;\eps_1);$
	\item formula~\eqref{RightLinearity} with the coefficients given in~\eqref{mrcup} holds on the domain
$$
\om_3(\eps;\tr) = \Set{(x,y)\in \omega_\eps}{b(2\eps) \leq \uR(x,y)\leq  \tr}.
$$
\end{itemize}
Assume that $b_\eps(x,y)$ is continuous as a function of $(x,y,\eps)$ on the domain
$$
\om=\Set{(x,y)\in\omega_\eps}{(x,y)\in\om_1(\eps;\tl)\cup\om_2(\eps;\eps_1)\cup\om_3(\eps;\tr),\;
\eps_1\leq\eps\leq\eps_2}.
$$ 	
Then\textup,
\begin{equation}
\Bell(x,y,z) = b_{\sqrt{1-z^2}}(x,y),\qquad \big(x,y,\sqrt{1-z^2}\big) \in \om.
\end{equation}
\end{Th}

\section{Case $f(t) = \chi_{_{[0, +\infty)}}(t)$ and sharp tail estimates}\label{S4}
In this section, we will present the proofs of Theorem~\ref{th_bellman_is_less_b_eps} and 
Corollary~\ref{cor_sharp_tail}. In other words, we will describe the trace of the Bellman 
function~\eqref{Bellman} with $f(t)=\chi_{_{[0,+\infty)}}(t)$ on~$\Omega_R$ defined in~\eqref{OmegaR}.

The exposition is organized as follows. We start with solving an auxiliary optimization problem, which we call 
the model problem, in Subsection~\ref{s41_model}. Subsection~\ref{s42_foliation} contains the proof of 
Theorem~\ref{th_bellman_is_less_b_eps}, the solution of the model problem from the previous subsection plays the 
crucial role there. Finally, we establish Corollary~\ref{cor_sharp_tail} in Subsection~\ref{s44_sharp}.

\subsection{Model problem}\label{s41_model} 
\subsubsection{Setting}
Consider the domain
\begin{equation}
\OmSmile = \{(x,y) \in \R^2 \,\big{|}\, 2x^2 \le y \le x^2+1,\ x\in [-1,1]\}.
\end{equation} 
We say that a function~$R\colon \OmSmile\to \mathbb{R}$ satisfies the main inequality of the 
\emph{model problem} provided
\begin{equation}\label{model_problem_dynamic}
\begin{gathered}
R(x,y) \ge \alpha_+ R(x_+, y_+) + \alpha_- R(x_-, y_-), \quad \text{where}
\\
x=\alpha_+ x_+ +\alpha_-x_-,\quad y=\alpha_+ y_+ +\alpha_-y_-,\quad\frac{y_+-y_-}{x_+-x_-}=2x,
\\
\alpha_+ +\alpha_-=1,\quad\alpha_0,\alpha_1\in(0,1),\quad\text{and}
\quad (x,y),\;(x_+,y_+),\;(x_-,y_-)\in\OmSmile,
\end{gathered}
\end{equation}
for any choice of the parameters. Geometrically, the main inequality of the model problem is 
the usual convexity condition when the point~$(x,y)$ splits into~$(x_+,y_+)$ and~$(x_-,y_-)$ 
along the tangent to the parabola~${\bf y} = {\bf x}^2 + c$ passing through~$(x,y)$.

We posit the model problem: find the pointwise minimal function~$\Bs$ among all function $R\colon\OmSmile\to\R$ 
that satisfy the main inequality of the model problem and the boundary conditions
\begin{equation}\label{BCModel}
R(x,2x^2) = h(x),\quad x\in [-1,1],\qquad \hbox{where}\ 
h(x) =
\begin{cases}
1,\quad x \ge 0;
\\
\frac{1}{2},\quad x < 0.
\end{cases}
\end{equation}

\begin{Rem}
One may consider a similar homogeneous extremal problem on a larger domain~$\{y\geq 2x^2\}$ 
\textup(with the same boundary value~$h$\textup). It is easy to see that the restriction 
to~$\OmSmile$ of the solution of this new problem coincides with~$\Bs$. Thus\textup,
\begin{equation}\label{Homogeneity}
\Bs(\lambda x, \lambda^2 y) = \Bs(x,y),\quad (x,y)\in\OmSmile,\ 0<\lambda \leq 1.
\end{equation} 
\end{Rem}

\subsubsection{Parametrization and differential equation}\label{s412_foliation}
The domain $\OmSmile$ can be split into the parabolic arcs
\begin{equation}
P_c = \big\{(x,x^2+c^2)\big| \, |x| \le c\big\}, \quad  0 \le c \le 1.
\end{equation} 
By the homogeneity relation~\eqref{Homogeneity}, it suffices to focus on the case~$c = 1$ and determine 
the values of~$\Bs$ on~$P_1$.

Consider a parametrization $(\uuu(t),\vvv(t))$ of the arc~$P_1$. More specifically, we consider 
two functions $\uuu$ and $\vvv=\uuu^2+1$ defined on $[1,+\infty]$ such that $\uuu$ increases, 
$\uuu(1) = -1$, and $\uuu(+\infty)=1$. We split every point~$(\uuu(t),\vvv(t)) \in P_1$ 
into~$(x_+(t), y_+(t))$ lying on the boundary $\{y = 2x^2, x\geq 0\}$ and an infinitesimally 
close point~$(x_-(t), y_-(t))$, according to the rules~\eqref{model_problem_dynamic}  
(see Figure~\ref{fig:model_problem}). We will search for the function~$\mmm\colon\OmSmile\to\mathbb{R}$ 
satisfying homogeneity relation~\eqref{Homogeneity} (with $\mmm$ instead of $\Bs$) for which 
the main inequality ``turns into equality'' along the said splitting. It will appear that 
the function $\mmm$ constructed in such a way satisfies the equation
\begin{equation}\label{grad_direction}
\mmm(\uuu,\vvv)+\Big\langle\nabla\mmm\big(\uuu,\vvv\big),\left(x_+-\uuu,y_+-\vvv\right)\Big\rangle=1. 
\end{equation}
Note that the parametrization has not been specified yet. It will be specified in 
Subsubsection~\ref{ss413} below. 

The trace of $\mmm$ on $P_1$ will be denoted by $\Psi$:
\eq{\label{PsiDefNew}
\Psi(x) = \mmm(x,x^2+1), \quad x \in [-1,1]; \qquad \mmm(x,y) = 
\Psi\Big(\frac{x}{\sqrt{y-x^2}}\Big), \quad (x,y) \in \OmSmile.
}

\begin{figure}
\begin{center}
\includegraphics[width=0.5\columnwidth]{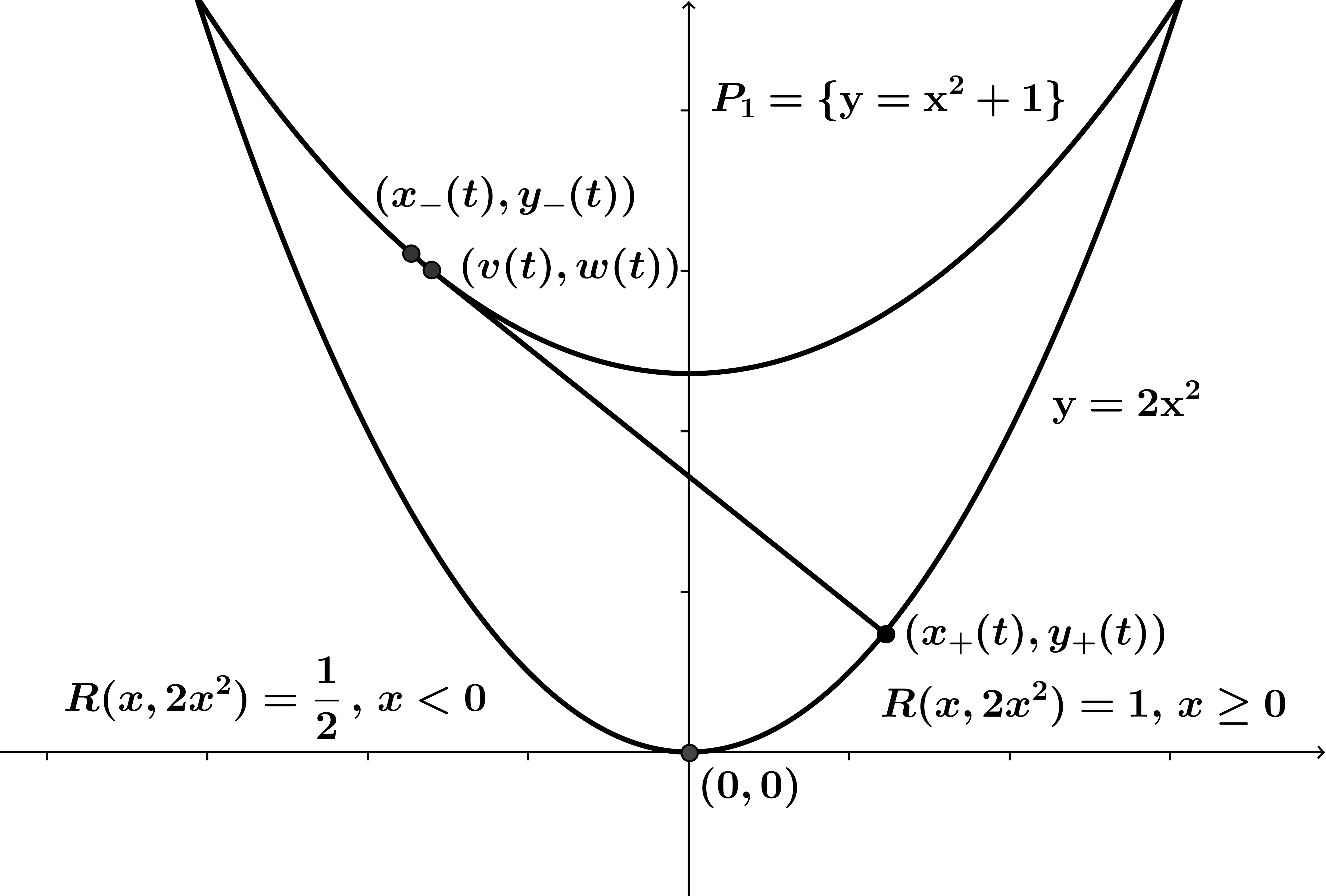}
\caption{Illustration to the model problem}
\label{fig:model_problem}
\end{center}
\end{figure}

Recall the boundary values~\eqref{BCModel}. We will search for the functions $\uuu,\vvv$, and $\Psi$ in the form 
\begin{equation}\label{U_V_parabola}
\uuu(t)=\frac1t\int\limits_0^t\varphi(s)\,ds,\qquad\vvv(t)=\frac1t\int\limits_0^t2\varphi^2(s)\,ds, 
\qquad\Psi(\uuu(t)) = \frac{1}{t} \int\limits_0^t h(\varphi(s))\,ds, \qquad  t\geq 1,
\end{equation}
where $\varphi$ satisfies the conditions
\begin{equation}
\varphi\colon [0,\infty)\to\mathbb{R},\qquad\varphi(t)=-1\,\text{ for }t\in[0,1);\qquad\varphi(1)=0; 
\qquad\varphi(t)\ge0\,\text{ for }t\ge 1.
\end{equation}
We also require $\varphi$ to be a non-decreasing function. One may check that $x_+(t) = \varphi(t)$, 
$y_+(t) = 2 \varphi^2(t)$ in the sense that the tangent vector to the curve~$(\uuu(t),\vvv(t))$ points 
to~$(\vf(t),2\vf^2(t))$, and
\begin{gather}
\uuu'(t)=-\frac{\uuu(t)}t+\frac{\vf(t)}t,\qquad \vvv'(t)=-\frac{\vvv(t)}t+\frac{2\vf^2(t)}t, 
\label{model_UVG_deriv}
\\
\label{eq411}\frac{d}{dt}\mmm(\uuu(t),\vvv(t)) = -\frac{\mmm(\uuu(t),\vvv(t))}{t}+\frac{h(\vf(t))}{t}. 
\end{gather}
We rewrite the left hand side of~\eqref{eq411}:
\begin{multline}\label{eq412}
\frac{d}{dt}\mmm(\uuu(t),\vvv(t)) = \langle\nabla\mmm(\uuu(t),\vvv(t)),(\uuu'(t),\vvv'(t))\rangle = 
\\
-\frac1t\langle\nabla\mmm(\uuu(t),\vvv(t)),(\uuu(t)-\varphi(t),\vvv(t)-2\vf(t)^2)\rangle,
\end{multline}
then plug~\eqref{eq412} into~\eqref{eq411} taking into account that~$h(\varphi(t)) = 1$ 
for~$t \geq 1$, and obtain~\eqref{grad_direction}.

\subsubsection{Solution of differential equation}\label{ss413}
Using~\eqref{U_V_parabola} and the relation~$\vvv = \uuu^2 + 1$, we can write down 
the following chain of equalities:
\begin{equation}
t(\uuu^2(t)+1) = t\vvv(t) = 2\int_0^t\vf^2(s)\,ds \stackrel{\scriptscriptstyle \eqref{model_UVG_deriv}}{=} 
2\int_0^t \big(\uuu(s)+s\uuu'(s)\big)^2 ds.
\end{equation}
We differentiate this relation and obtain
\begin{equation*}
\uuu^2(t)+1+2t\uuu(t)\uuu'(t) = 2(\uuu(t)+t\uuu'(t))^2,
\end{equation*}
or
$$
(\uuu(t)+2t\uuu'(t))^2 = 2-\uuu^2(t).
$$
We are looking for increasing functions $\varphi$ and $\uuu$.
Thereby, we have to solve the following Cauchy problem
\begin{equation*}
2t\uuu'(t) = -\uuu(t)+\sqrt{2-\uuu^2(t)},\qquad \uuu(1)=-1,\qquad t\geq1,
\end{equation*}
or
$$
\frac{dt}t=\frac{2d\uuu}{\sqrt{2-\uuu^2}-\uuu}\,.
$$
Hence
\begin{multline*}
\log t = \int_{-1}^\uuu \frac{2}{\sqrt{2-z^2}-z}\,dz\;=\int_{-\pi/4}^{\arcsin(\frac{\uuu}{\sqrt2})} 
\!\!\frac{2}{\cos\theta-\sin\theta}\cos\theta\,d\theta = 
\\
\Big(\theta-\log(\cos\theta-\sin\theta)\Big)\Big|_{-\pi/4}^{\arcsin(\frac{\uuu}{\sqrt{2}})}
=\arcsin\left(\frac{\uuu}{\sqrt{2}}\right)+\frac{\pi}{4}+\frac{\log 2}{2} 
- \log\left(\sqrt{1-\frac{\uuu^2}2}-\frac{\uuu}{\sqrt{2}}\right).
\end{multline*}
Therefore,
\begin{equation}
\label{310101}
t=\frac2{\sqrt{2-\uuu^2}- \uuu}\,e^{\arcsin\left(\frac{\uuu}{\sqrt{2}}\right)+\frac{\pi}{4}}.
\end{equation}
Note that $t$ runs from $1$ to $+\infty$ as $\uuu$ runs from $-1$ to $1$.

Now, we are able to compute $\Psi(\uuu(t))$. Recall that $\varphi(t)=-1$, when $t\in[0,1)$, and 
$\varphi(t)>0$ for $t>1$. Therefore, for $t\in[0,1)$ we have
$$
\Psi(\uuu(t)) = \mmm(-1,2) = \frac12\,.
$$
For $t>1$ we use the last formula in \eqref{U_V_parabola} and the definition of the function 
$h$ in~\eqref{BCModel} and deduce
\begin{equation}\label{Bs_by_t}
\Psi(\uuu(t))=\frac1t\int\limits_0^1\frac12\,ds + \frac1t\int\limits_1^t 1\,ds = 1-\frac1{2t}.    
\end{equation}
If we plug here the solution found in~\eqref{310101}, we get
\begin{equation}\label{psi_def}
\Psi(\uuu)=1-\frac{\sqrt{2-\uuu^2}-\uuu}4e^{-\arcsin\left(\frac{\uuu}{\sqrt{2}}\right)-\frac{\pi}{4}},
\qquad\uuu\in[-1,1].
\end{equation}
By the homogeneity relation~\eqref{PsiDefNew} for an arbitrary point $(x,y)\in\OmSmile$ we have 
\begin{equation}\label{model_G_def}
\mmm(x,y) =  1-\frac{\sqrt{1-\ttt^2}-\ttt}{2\sqrt{2}}e^{-\arcsin \ttt-\frac{\pi}{4}}, 
\quad \text{ where }\ \ttt = \ttt(x,y) = \frac{x}{\sqrt{2(y-x^2)}}.
\end{equation}
We have finished the construction of the function $\mmm$ and now will prove 
that it solves the model problem. 

\subsubsection{Verification of the main inequality}

We would like to prove that the function~$\mmm$ defined in~\eqref{model_G_def} satisfies 
the main inequality~\eqref{model_problem_dynamic} of the model problem. We will not do this directly, 
but rather rely upon a principle similar to Lemma~\ref{main_under_tangent_Le}. We omit the proof 
of the following lemma because it is completely similar to the proof of Lemma~\ref{main_under_tangent_Le}. 

\begin{Le}\label{under_tangent_Le} Assume~$\Gw\colon\OmSmile\to\mathbb{R}$ is differentiable 
on~$\OmSmile$ and satisfies the inequality 
\begin{equation}\label{model_problem_dynamics_dif_form}
\Gw(\xx, \yy) \le \Gw(x_0,y_0) + \frac{\partial \Gw}{\partial x}(x_0,y_0)\cdot (\xx - x_0) + 
\frac{\partial \Gw}{\partial y}(x_0,y_0) \cdot(\yy - y_0)
\end{equation}
for every points $(x_0, y_0)\in\OmSmile$ and $(\xx,\yy)\in\OmSmile$ such that $\yy-y_0=2x_0(\xx-x_0)$. 
Then\textup,~$\Gw$ satisfies the main inequality of the model problem~\eqref{model_problem_dynamic}.
\end{Le}

\begin{Le}\label{VerifMainIneq}
The function~$\mmm$ as in~\eqref{model_G_def} satisfies~\eqref{model_problem_dynamics_dif_form} in the 
role of~$R$\textup, i.\,e.\textup, the inequality 
\begin{equation}\label{B_under_tangent}
\mmm(\xx, \yy) \le \mmm(x_0,y_0) +  \frac{\partial \mmm}{\partial x}(x_0,y_0)\cdot (\xx -x_0)+ 
\frac{\partial \mmm}{\partial y}(x_0,y_0)\cdot (\yy-y_0)
\end{equation}
holds true for any~$(x_0,y_0)\in\OmSmile$ and~$(\xx,\yy)\in\OmSmile$ such that~$\yy-y_0=2x_0(\xx-x_0)$.
\end{Le}

\begin{proof}
{\bf Case $\xx > x_0$.} 

Let $c\in (0,1]$. For any point $(x,y)$ such that~$(x,y)\in\OmSmile$ and $\max(-2cx,x^2)\leq y\leq x^2+c^2$ 
we find two numbers $\ur$ and $\vr$ such that $\vr\leq x\leq \ur$ and 
$$
\frac{2\ur^2 - \vr^2 - c^2}{\ur-\vr} = \frac{y-\vr^2-c^2}{x-\vr} = 2\vr,
$$
see Figure~\ref{fig:UVpoints}.
\begin{figure}
    \centering	
    \includegraphics[width=0.5\columnwidth]{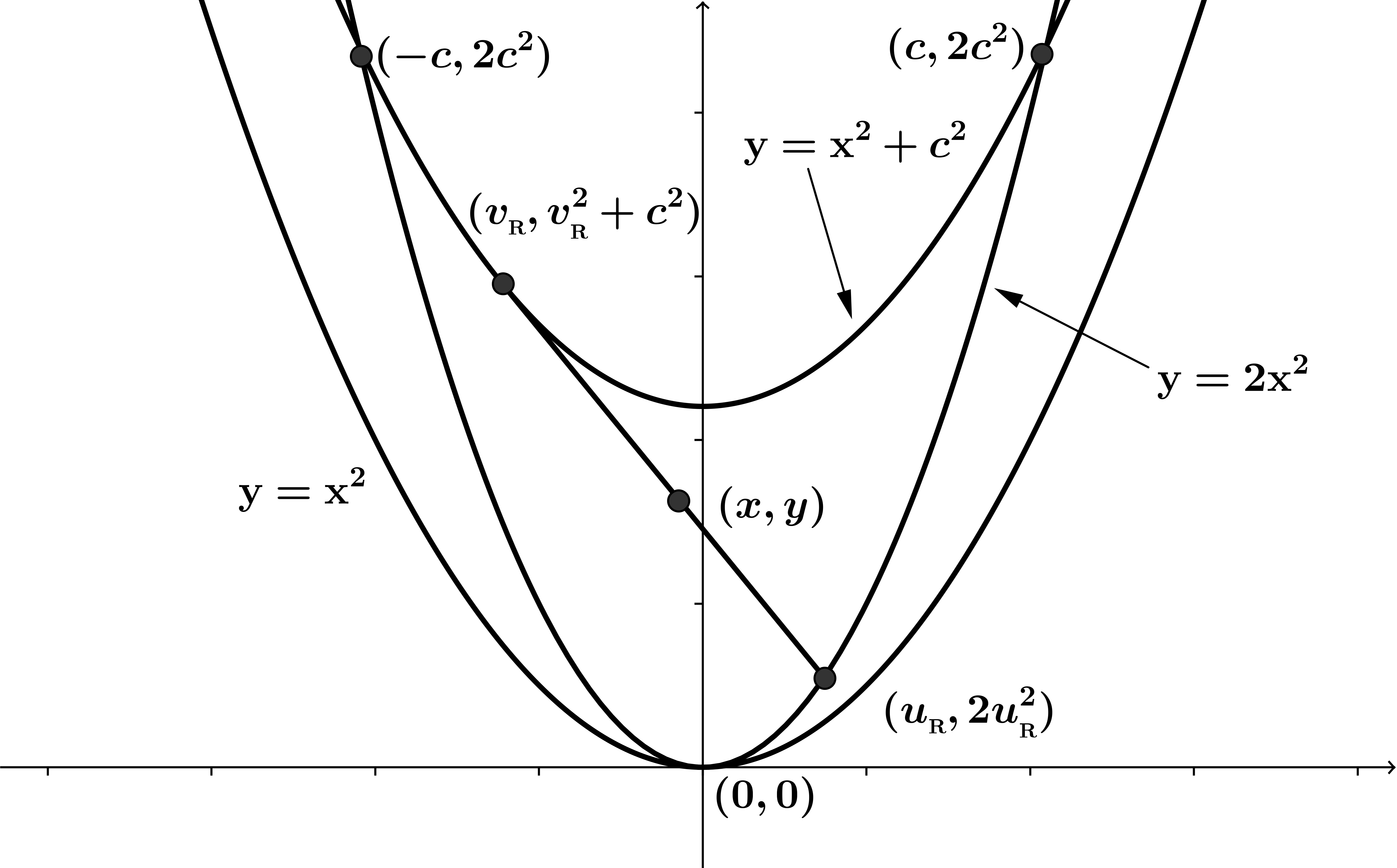}
    \caption{Definition of $\ur$ and $\vr$.}
    \label{fig:UVpoints}
\end{figure}

We deduce that
\begin{equation}\label{u_v_x1_x2_def}
\vr = \vr(x,y,c) = x - \sqrt{x^2+c^2-y},\qquad \ur= \ur(x,y,c) = \frac{\vr+\sqrt{2c^2-\vr^2}}{2}.
\end{equation}

We introduce the function $\GR$ defined on the domain 
$$
\Big\{(x,y,c)\,\Big|\; c\in(0,1], \, \max(-2cx,x^2)\leq y\leq x^2+c^2\Big\}
$$ 
by the formula:
\begin{equation}\label{G_c_def}
\GR(x,y,c) = \frac{\ur-x}{\ur-\vr}\Psi\left(\frac{\vr}{c}\right) + \frac{x-\vr}{\ur-\vr} = 
\frac{\ur-x}{\ur-\vr}\left(\Psi\left(\frac{\vr}{c}\right)-1\right)+1,
\end{equation}
where $\ur = \ur(x,y,c)$, $\vr = \vr(x,y,c)$, and~$\Psi$ was defined in~\eqref{PsiDefNew} and got 
its explicit form in~\eqref{psi_def}. For any $c$ fixed the function $\GR(\,\cdot\,,\,\cdot\,,c)$ 
is linear along each segment connecting the points $(\vr,\vr^2+c^2)$ and $(\ur,2\ur^2)$, 
and coincides with $\mmm$ at their endpoints. Also from the construction of $\mmm$ 
(see formula~\eqref{grad_direction}) we deduce that the function $\GR$ has the following property:
\begin{equation}
\mmm(x_0,y_0) +  \frac{\partial \mmm}{\partial x}(x_0,y_0)\cdot (\xx -x_0)+ 
\frac{\partial \mmm}{\partial y}(x_0,y_0)\cdot 2x_0 (\xx - x_0) = \GR(\xx,\yy,c_0)
\end{equation}
with $c_0 = \sqrt{y_0-x_0^2}$.

On the other hand, for $\cc = \sqrt{\yy-\xx^2}$ we have $\vr(\xx,\yy,\cc)=\xx$, and therefore
$\GR(\xx,\yy,\cc) = \mmm(\xx,\yy)$. It may be  seen that $\cc < c_0$:
$$
\cc^2=\yy-\xx^2=y_0-x_0^2-(\xx-x_0)^2<y_0-x_0^2=c_0^2\,.
$$
Thus, to prove~\eqref{B_under_tangent} it suffices to show that the function $\GR(\xx,\yy,c)$ 
does not decrease in~$c$. In other words, we wish to verify the inequality 
\begin{equation}\label{MonotonicityGR}
\frac{\partial \GR(x,y,c)}{\partial c} \geq 0.
\end{equation}

It follows from \eqref{u_v_x1_x2_def} that
\begin{equation}\label{v_c}
\frac{\partial \vr}{\partial c} = - \frac{c}{\sqrt{x^2+c^2-y}} = \frac{c}{\vr-x},
\end{equation}
\begin{equation}\label{u_c}
\frac{\partial \ur}{\partial c} = \frac{1}{2} \left( \frac{\partial \vr}{\partial c} +
\frac{4c - 2\vr \frac{\partial \vr}{\partial c} }{2\sqrt{2c^2 -\vr^2}} \right) = 
\frac{1}{2} \left(\frac{c}{\vr-x} + \frac{2c-\frac{c\vr}{\vr-x}}{2\ur-\vr} \right)=
\frac{c(\ur-x)}{(\vr-x)(2\ur-\vr)}.
\end{equation}
We differentiate~\eqref{G_c_def}  and obtain 
\begin{equation*}
\frac{\partial \GR(x,y,c)}{\partial c} = \frac{\frac{\partial \ur}{\partial c}(\ur-\vr) - 
(\frac{\partial \ur}{\partial c}-\frac{\partial \vr}{\partial c})(\ur-x)}{(\ur-\vr)^2}
\left(\Psi\left(\frac{\vr}{c}\right) - 1 \right) + 
\frac{\ur-x}{\ur-\vr}\cdot\frac{\frac{\partial \vr}{\partial c} c - \vr}{c^2}\Psi'\left(\frac{\vr}{c}\right).
\end{equation*}
Formula \eqref{psi_def} for the function $\Psi$ implies
\begin{equation}\label{d_psi}
\Psi'(\uuu) = \frac{1}{2} e^{-\pi/4 - \arcsin\frac{\uuu}{\sqrt{2}}}.   
\end{equation}{}
Using~\eqref{psi_def} and \eqref{d_psi}, we continue the evaluation of 
$\frac{\partial \GR(x,y,c)}{\partial c}$ and obtain that it equals to
\begin{equation*}
\left(\frac{\frac{\partial\ur}{\partial c}(\ur-\vr-\ur+x)+\frac{\partial\vr}{\partial c}(\ur-x)}{4(\ur-\vr)^2}
\left(\frac{\vr}{c} - \sqrt{2-\frac{\vr^2}{c^2}} \right) + 
\frac{\ur-x}{\ur-\vr}\cdot\frac{\frac{\partial \vr}{\partial c} c - \vr}{2c^2} \right)
e^{-\arcsin\big(\frac{\vr}{\sqrt{2}c}\big)-\frac{\pi}{4}}. 
\end{equation*}
We may omit the exponent multiplier since we are interested in the sign of the expression 
$\frac{\partial \GR(x,y,c)}{\partial c}$ only.
Note that relation~\eqref{u_v_x1_x2_def} yields 
\begin{equation*}
\frac{\vr}{c} - \sqrt{2-\frac{\vr^2}{c^2}} = 2 \frac{v_R - u_R}{c}. 
\end{equation*}
Applying~\eqref{v_c} and  \eqref{u_c}, we continue the computation
\begin{align*}
\frac{1}{(\ur-\vr)^2} &\left( (x-\vr)\frac{c(\ur-x)}{(\vr-x)(2\ur-\vr)} + 
\frac{c(\ur-x)}{\vr-x}\right) \frac{2(\vr-\ur)}{4c} + 
\frac{\ur-x}{2c^2(\ur-\vr)}\left(\frac{c^2}{\vr-x} - \vr\right) 
\\
&=\rule{0pt}{22pt}\frac{\ur -x}{2(\ur-\vr)^2} \left(\frac{1}{\vr -2\ur}+\frac{1}{\vr-x}\right)(\vr-\ur) 
-\frac{\ur-x}{2(\vr-\ur)}\left(\frac{1}{\vr-x} - \frac{\vr}{c^2}\right)
\\
&=\rule{0pt}{22pt}\frac{\ur-x}{2(\vr-\ur)} \left( \frac{1}{\vr-2\ur} + \frac{\vr}{c^2}\right)
=\frac{\ur-x}{2c^2(\ur-\vr)(2\ur-\vr)} \big( (\vr-\ur)^2 + c^2 - \ur^2\big), 
\end{align*}
which is non-negative because $0 \leq \ur \le c$ and $\vr\leq x\leq \ur$. This finishes the proof 
of~\eqref{MonotonicityGR}.

\bigskip
{\bf Case~$\xx < x_0$.}

We will construct another auxiliary function $\GL$ in the following way. Let $c>0$. For any point $(x,y)$ such that~$(x,y)\in \OmSmile$ and~$\max(2c x, x^2)\leq y \leq x^2+c^2$, we 
find two numbers $\ul$ and $\vl$ such that $x \leq \vl \leq \ul$ and 
$$
\frac{2\ul^2 - \vl^2 - c^2}{\ul-\vl} = \frac{ y-\vl^2-c^2}{ x-\vl} = 2\vl,
$$
see Figure~\ref{fig:UVL}.
\begin{figure}
    \centering
    \includegraphics[width=0.5\columnwidth]{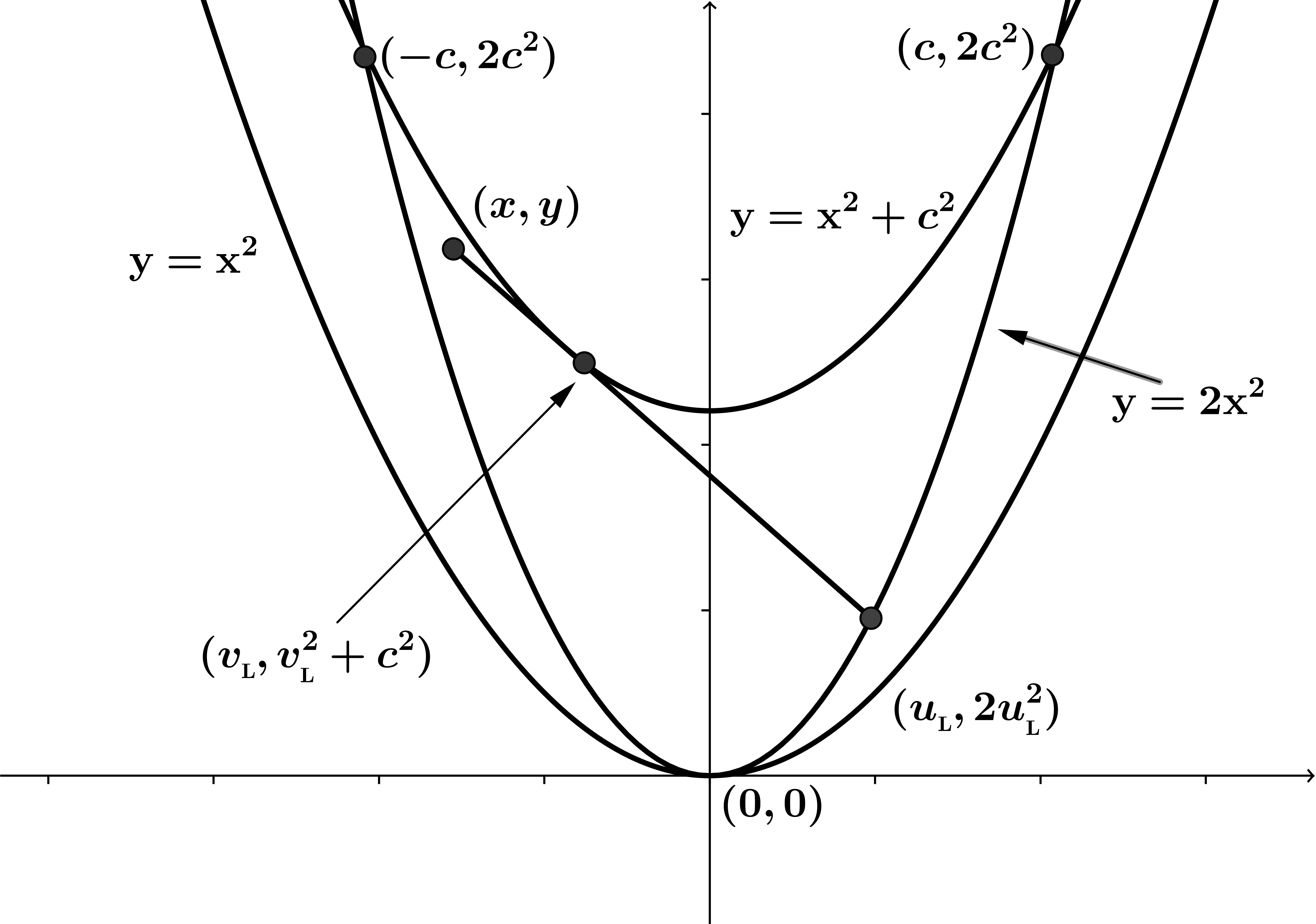}
    \caption{Definition of $\ul$ and $\vl$.}
    \label{fig:UVL}
\end{figure}
After some calculations, we get
\begin{equation}\label{u_v_x1_x2_defL}
\vl = \vl(x,y,c) =  x+\sqrt{x^2+c^2-y},\qquad\ul=\ul(x,y,c) = \frac{\vl+\sqrt{2c^2-\vl^2}}{2}.
\end{equation}

We introduce the function $\GL$ defined on the domain 
$$
\Big\{(x,y,c) \,\Big|\; c\in(0,1], \, \max(2cx,x^2)\leq y\leq x^2+c^2\Big\}
$$ 
by the formula:
\begin{equation}\label{GL_c_def}
\GL(x,y,c) = \frac{\ul-x}{\ul-\vl}\Psi\left(\frac{\vl}{c}\right) + \frac{x-\vl}{\ul-\vl} = 
\frac{\ul-x}{\ul-\vl}\left(\Psi\left(\frac{\vl}{c}\right)-1\right)+1.
\end{equation}
For any $c$ fixed the function $\GL(\,\cdot\,,\,\cdot\,,c)$ is linear on the extension of the segment 
connecting the points $(\vl,\vl^2+c^2)$ and $(\ul,2\ul^2)$ beyond the point $(\vl,\vl^2+c^2)$, 
and coincides with $\mmm$ at these two points. Also from the construction of $\mmm$ 
(see~\eqref{grad_direction}) we deduce that the function $\GL$ satisfies the following property:
\begin{equation}
\mmm(x_0,y_0) + \frac{\partial \mmm}{\partial x}(x_0,y_0)\cdot (\xx -x_0)+
\frac{\partial \mmm}{\partial y}(x_0,y_0)\cdot 2x_0 (\xx - x_0) = \GL(\xx,\yy,c_0), 
\end{equation}
for $c_0 = \sqrt{y_0-x_0^2}$. 

On the other hand, for $\cc = \sqrt{\yy-\xx^2}$ we have $\GL(\xx,\yy,\cc) = \mmm(\xx,\yy)$. 
Again, we have $\cc < c_0$. Thus, it suffices to show that the function $\GL(\xx,\yy,c)$ increases 
in~$c$, i.\,e., the inequality
\begin{equation}\label{eq250101}
\frac{\partial \GL(x,y,c)}{\partial c} \geq 0.
\end{equation}

From equations \eqref{u_v_x1_x2_defL} we obtain
\begin{equation}\label{v_cL}
\frac{\partial \vl}{\partial c} = \frac{c}{\sqrt{x^2+c^2-y}} = \frac{c}{\vl-x}\,.
\end{equation}
We note that the right hand side of~\eqref{v_cL} coincides with the right hand side of~\eqref{v_c} 
($\vr$ is simply replaced with $\vl$), and $\ul$ is defined by $\vl$ exactly in the same way as $\ur$ 
was defined by $\vr$. The same calculations as we have already done in the proof of~\eqref{MonotonicityGR} 
lead us to the fact that~\eqref{eq250101} is equivalent to  
\begin{equation}
\frac{\ul-x}{c^2(\ul-\vl)(2\ul-\vl)} \big((\vl-\ul)^2+ c^2-\ul^2\big) \ge 0, 
\end{equation}
which holds true because $x \leq \vl \leq \ul$ and $0 \leq \ul \leq c$.
\end{proof}

\subsubsection{Minimality}

Lemmas~\ref{under_tangent_Le} and~\ref{VerifMainIneq} imply~$\mmm(x,y)\geq\Bs(x,y)$ for any~$(x,y)\in\OmSmile$. 
Now we wish to prove the reverse inequality.

\begin{Le}\label{model_tangent_lemma}
Let~$\Bs$ be the solution of the model problem and let~$\mmm$ be the function defined in~\eqref{model_G_def}. 
Then\textup, for any~$(x,y)\in \OmSmile,$ we have~$\Bs(x,y) \geq \mmm(x,y)$.
\end{Le}

\begin{proof}
Due to the homogeneity relation~\eqref{Homogeneity}, it suffices to consider the case~$(x,y)\in P_1$. 
Let~$\rrr(t) = \Bs(\uuu(t),\vvv(t))$, here we use the parametrization $(\uuu(t),\vvv(t))$ of $P_1$ 
introduced in Subsubsection~\ref{s412_foliation}. We will show that $\rrr$ is continuous and for every 
$t\ge 1$ the inequality
\begin{equation}\label{G_GW_ineq_t}
\frac{d_-}{dt}[t\ \rrr(t)] \ge 1
\end{equation}
holds true. By~$\frac{\ d_-}{dt}$ we mean the lower derivative, that is
\begin{equation*}
\frac{d_- r}{dt}(t_0) =\liminf\limits_{t\to t_0}\frac{r(t) - r(t_0)}{t-t_0}.
\end{equation*}

Once~\eqref{G_GW_ineq_t} is proved, we may use formula~\eqref{Bs_by_t} that implies
\eq{   
\frac{d_-}{dt}\left[t \Big(\rrr(\uuu(t)) - \Psi(\uuu(t))\Big)\right]  \ge 0.
}
This yields the desired estimate~$\Bs(\uuu(t),\vvv(t)) - \mmm(\uuu(t),\vvv(t)) \geq 0$, 
because the two functions in question are continuous and are equal at $t=1$.

The proof of~\eqref{G_GW_ineq_t} and the continuity of $\rrr$ will take some time.  
Fix a point $(\uuu_0,\vvv_0)$, where $\vvv_0=\uuu_0^2+1$ and $\uuu_0=\uuu(t_0)$ for some $t_0 \in(1,\infty)$.
Draw the tangent line through $(\uuu_0,\vvv_0)$ to the upper boundary of~$\OmSmile$:
\eq{\label{090201}
y=2\uuu_0x+1-\uuu_0^2.
}
Take two more points on this line $(x_\pm,y_\pm)$, where one of them is the right point of intersection
with the lower boundary of $\OmSmile$, i.\,e., 
$$
x_+=\vf(t_0)=\frac{\uuu_0+\sqrt{2-\uuu_0^2}}2\,,\qquad y_+=2\vf^2(t_0)=1+\uuu_0\sqrt{2-\uuu_0^2}\,,
$$
and the second is defined as follows:
\begin{equation}
\label{060201}
x_-=\frac{\uuu\uuu_0+\sqrt{1+\uuu^2-\uuu_0^2}}{1+\uuu^2}\uuu\,,
\qquad y_-=\frac{\big(\uuu\uuu_0+\sqrt{1+\uuu^2-\uuu_0^2}\big)^2}{1+\uuu^2},
\end{equation}
where $\uuu=\uuu(t)$ for some $t\in[1,\infty)$, $t\ne t_0$.
At these points we have $\Bs(x_+,y_+)=1$ and $\Bs(x_-,y_-)=\rrr(t)$. The latter identity holds true 
by~\eqref{Homogeneity}, because the points $(x_-,y_-)$ and~$(\uuu,\uuu^2+1)$ lie on the parabola
$$
y=\frac{1+\uuu^2}{\uuu^2}x^2.
$$

We write down the concavity property~\eqref{model_problem_dynamic}:
\begin{align}
\label{050201}
\rrr(t_0)\geq\frac{x_+-\uuu_0}{x_+-x_-}\cdot\rrr(t)+\frac{\uuu_0-x_-}{x_+-x_-}\cdot1,&
\qquad\text{if }\ t<t_0\,;
\\
\label{050202}
\rrr(t)\geq\frac{x_+-x_-}{x_+-\uuu_0}\cdot\rrr(t_0)+\frac{x_--\uuu_0}{x_+-\uuu_0}\cdot1,&
\qquad\text{if }\ t>t_0\,.
\end{align}
We may rewrite~\eqref{050201} and~\eqref{050202}
as follows
\begin{align*}
(x_+-\uuu_0)(\rrr(t)-\rrr(t_0))\le(x_--\uuu_0)(1-\rrr(t_0)),&\qquad\text{if }\ t<t_0\,;
\\
(x_+-\uuu_0)(\rrr(t)-\rrr(t_0))\ge(x_--\uuu_0)(1-\rrr(t_0)),&\qquad\text{if }\ t>t_0\,.
\end{align*}
Both these inequalities turn into
\begin{equation}
\label{050203}
\frac{\rrr(t)-\rrr(t_0)}{t-t_0}\ge\frac{x_--\uuu_0}{t-t_0}\cdot\frac{1-\rrr(t_0)}{x_+-\uuu_0}
\end{equation}
after dividing by $(t-t_0)(x_+-\uuu_0)$.

Let us calculate the right hand side of this inequality. Using~\eqref{model_UVG_deriv} we get
$$
x_+-\uuu_0=\vf(t_0)-\uuu(t_0)=t_0\uuu'(t_0)\,.
$$
From the definition~\eqref{060201} of $x_-$ we deduce
$$
x_--\uuu_0=\frac{\uuu^2-\uuu_0^2}{\uuu_0+\uuu\sqrt{1+\uuu^2-\uuu_0^2}}\,.
$$
Therefore,~\eqref{050203} may be rewritten as
\eq{\label{090202}
\frac{\rrr(t)-\rrr(t_0)}{t-t_0}\ge\frac{\uuu-\uuu_0}{t-t_0}\cdot
\frac{\uuu+\uuu_0}{\uuu_0+\uuu\sqrt{1+\uuu^2-\uuu_0^2}}\cdot
\frac{1-\rrr(t_0)}{t_0\uuu'(t_0)}\,.
}
We see that the right hand side has a limit as $t\to t_0$:
$$
\lim_{t\to t_0}\frac{\uuu-\uuu_0}{t-t_0}\cdot\frac{\uuu+\uuu_0}{\uuu_0+\uuu\sqrt{1+\uuu^2-\uuu_0^2}}\cdot
\frac{1-\rrr(t_0)}{t_0\uuu'(t_0)}=\frac{1-\rrr(t_0)}{t_0}\,,
$$
whence
$$
\liminf_{t\to t_0}\frac{\rrr(t)-\rrr(t_0)}{t-t_0}\ge\frac{1-\rrr(t_0)}{t_0}\,,
$$
what is exactly the desired estimate~\eqref{G_GW_ineq_t}.

It remains to check continuity of $\rrr$. First we note that~\eqref{090202} implies that $\rrr$ 
is an increasing function because $\uuu$ is. We will write down the same property~\eqref{model_problem_dynamic} 
with~$(x_+,y_+)$ being not the right but the left point of intersection
with the lower boundary of~$\OmSmile$, i.\,e., 
$$
x_+=\frac{\uuu_0-\sqrt{2-\uuu_0^2}}2.
$$
The point $(x_-,y_-)$ is defined as before by~\eqref{060201}, where $\uuu=\uuu(t)$ and $t>t_0$. Thus, now we 
have $x_+<\uuu_0<x_-$, $\Bs(x_+,y_+)=\tfrac12$, $\Bs(x_-,y_-)=\rrr(t)$, and the concavity 
property~\eqref{model_problem_dynamic} takes the form
$$
\rrr(t_0)\geq\frac{x_--\uuu_0}{x_--x_+}\cdot\tfrac12+\frac{\uuu_0-x_+}{x_--x_+}\cdot\rrr(t)\,.
$$
Therefore,
\eq{\label{090203}
(x_--\uuu_0)(\rrr(t_0)-\tfrac12)\geq(\uuu_0-x_+)(\rrr(t)-\rrr(t_0))\,,
}
whence
\begin{align*}
0\leq\rrr(t)-\rrr(t_0)&\leq\frac{x_--\uuu_0}{\uuu_0-x_+}(\rrr(t_0)-\tfrac12)\leq
\frac{x_--\uuu_0}{\uuu_0-x_+}\cdot\tfrac12
\\
&=\frac{\uuu^2-\uuu_0^2}{\uuu_0+\uuu\sqrt{1+\uuu^2-\uuu_0^2}}\cdot\frac1{\uuu_0+\sqrt{2-\uuu_0^2}}
\leq\frac{\uuu-\uuu_0}{\uuu_0+\sqrt{2-\uuu_0^2}}.
\end{align*}
Since $\uuu$ is a continuous function, this inequality proves that $\rrr$ is continuous at 
any point $t_0$, $t_0>1$ (i.\,e., $\uuu_0>-1$). It remains to verify continuity of the function 
$\rrr$ at the point $t=1$ from the right. We know that $\rrr$ is an increasing function, therefore, 
there exists a limit
\eq{
\lim_{t\to1+}\rrr(t)\df\rrr_1.
}
Due to~\eqref{090203} we have $\rrr(t_0)\ge\tfrac12$ for every $t_0>1$, whence $\rrr_1\geq\tfrac12$. 
At the same time we have already proved that $\rrr(t)\leq\Psi(\uuu(t))=1-\tfrac1{2t}$, i.\,e., 
$\rrr_1\leq\tfrac12$. So, $\rrr_1=\tfrac12$ and we have proved continuity of $\rrr$ on~$[1,\infty)$.
\end{proof}

Summarizing all preceding consideration, we conclude that the solution of the model problem is 
given by the following formula:
\begin{equation}
\Bs(x,y) =  1-\frac{\sqrt{1-\ttt^2}-\ttt}{2\sqrt{2}}e^{-\arcsin \ttt-\frac{\pi}{4}}, \quad 
\text{ where }\ \ttt = \ttt(x,y) = \frac{x}{\sqrt{2(y-x^2)}}.
\end{equation}

\subsection{Construction of the function and verification of the main inequality}\label{s42_foliation}
The set~$\Omega_R$ defined in~\eqref{OmegaR} has special relationship with the splitting 
rules~\eqref{SplittingRules}. It follows from Lemma~\ref{SimpleGeometry} that if~$(x,y,z)\in\Omega_R$ 
is split into some points~$(x_j,y_j,z_j)$ according to the rules~\eqref{SplittingRules}, then, first, 
all the points~$(x_j,y_j)$ lie on the tangent line to the parabola~${\bf y}-{\bf x}^2 = 1-z^2$, 
and second, all the points~$(x_j,y_j,z_j)$ belong to~$\Omega_R$. 
One may say that~$\Omega_R$ has separate dynamics. 
 
Thus, if we denote~$\Bell(x,y,\sqrt{1-y+x^2})$ by~$\BU(x,y)$, then~$\BU$ may be described as the minimal 
among functions~$G \colon \omega_1\to\mathbb{R}$ that satisfy the boundary conditions
$G(x,x^2) = \chi_{_{[0,\infty)}}(x)$ and the main inequality
\begin{equation}\label{roof_problem_dynamic}
\begin{gathered}
G(x,y) \ge \sum\limits_{j=1}^N\alpha_j G(x_j, y_j), \quad \text{where}
\\
x=\sum\limits_{j=1}^N\alpha_j x_j,\qquad y=\sum\limits_{j=1}^N\alpha_j y_j,\qquad\frac{y_j-y}{x_j-x}=2x,
\\
\sum\limits_{j=1}^N\alpha_j = 1,\qquad \alpha_j\geq 0, \qquad \text{and}\quad(x,y),\;(x_j,y_j)\in\omega_1.
\end{gathered}
\end{equation}
Note that the main inequality (or the splitting rules) almost coincides with the main 
inequality~\eqref{model_problem_dynamic} of the model problem. The only difference is that 
the two extremal problems are set on different domains (the splitting into~$N$ points with 
arbitrary~$N$ may be reduced to many splittings into~$2$ points; formally, we will not use this principle).

\begin{Le}\label{Le270101}
The function $\BU$ satisfies the following equality\textup:
\begin{equation*}
\BU(x,y)=b_{\sqrt{y-x^2}}(x,y),\qquad x^2 \leq y \leq \min(2x^2, x^2+1). 
\end{equation*}
\end{Le}

\begin{proof}
Lemma~\ref{SimpleMajorant} implies
\eq{
\BU(x,y)\leq b_{\sqrt{y-x^2}}(x,y),\qquad x^2 \leq y \leq \min(2x^2, x^2+1),
}
therefore, it suffices to prove the reverse inequality. Note that here we cannot use theorems 
from Section~\ref{S3} directly due to the discontinuity of~$f$. 

First, let $x^2 \leq y \leq \min(2x^2, x^2+1)$ and $x>0$. Then, by~\eqref{popolam} we have
$\BU(x,y)\geq 1$.

Second, let $(x_0,y_0)\in\omega_1$ with $x_0^2\leq y_0\leq2x_0^2$ and $x_0<0$. Let $\eps=\sqrt{y_0-x_0^2}$. 
Take any small $\theta>0$ and apply Remark~\ref{Rem270101} with $(x,y)=(-\eps+\theta,(-\eps+\theta)^2+\eps^2)$:
\begin{multline*}
\BU(x_0,y_0) \geq e^{-\frac{x-x_0}{\eps}}\liminf_{\delta\to 0+} \BU(x,y-\delta) \Geqref{popolam}
\\  
e^{-\frac{x-x_0}{\eps}}\liminf_{\delta\to 0+} \frac{f(x-\sqrt{y-\delta-x^2})+f(x+\sqrt{y-\delta-x^2})}{2}=
\frac{1}{2}e^{-\frac{x-x_0}{\eps}}.
\end{multline*}
Considering arbitrarily small~$\theta>0$, we obtain 
\eq{
\BU(x_0,y_0)\geq\frac12e^{\frac{x_0}{\eps}+1}\stackrel{\scriptscriptstyle\eqref{b_eps_def}}{=}b_\eps(x_0,y_0).
}
\end{proof}

Lemma~\ref{Le270101} implies that $\BU(x,y) = b_{\sqrt{y-x^2}}(x,y)$ on $\omega_1\setminus\OmSmile$. 
Moreover, $\BU|_{\OmSmile}$ satisfies the boundary conditions~\eqref{BCModel} and the main 
inequality~\eqref{model_problem_dynamic} of the model problem. Thus,
\begin{equation}
\BU(x,y) \geq \Bs(x,y),\quad (x,y)\in \OmSmile.
\end{equation} 
To prove Theorem~\ref{th_bellman_is_less_b_eps}, it suffices to show that the function~$G$ defined as
\begin{equation}
G(x,y) = 
\begin{cases}
b_{\sqrt{y-x^2}}(x,y),\quad &(x,y) \in \omega_1 \setminus \OmSmile,\\
\Bs(x,y),\quad &(x,y)\in \OmSmile,
\end{cases}
\end{equation}
satisfies the main inequality~\eqref{roof_problem_dynamic}. This is our target for the remaining 
part of the subsection. It is convenient to introduce the domains
\begin{equation}
\OmR=\{(x,y)\in\omega_1\mid y\le 2x^2,\,x\ge 0\};\qquad\OmL=\{(x,y)\in\omega_1\mid y\le 2x^2,\,x\le 0\}.
\end{equation}
The function $G$ is homogeneous: 
$G(\lambda {\bf x}, \lambda^2{\bf y}) = G({\bf x},{\bf y})$ for $\lambda \in (0,1]$, therefore, without loss of generality 
we may assume that $y=x^2+1$. If $(x,y) \notin \OmSmile$ then
$$
G(x,y)=\bbb{x}{y}\Geqref{ChainOfInequalities}\sum\limits_{j=1}^N\alpha_j\bbb{x_j}{y_j} 
\geq \sum\limits_{j=1}^N\alpha_j G(x_j,y_j).
$$
In what follows we consider only $(x,y) \in \OmSmile$ such that $ y \ne 2x^2$. 
Instead of verifying~\eqref{roof_problem_dynamic} we will prove the inequality
\begin{equation}\label{eq280101}
G(\xx,\yy)\le G(x,y)+\frac{\partial G}{\partial x}(x,y)\cdot(\xx-x)+\frac{\partial G}{\partial y}(x,y)\cdot(\yy-y)
\end{equation}
for $(\xx,\yy)\in\omega_1\cap\ell$, where $\ell = \{({\bf x},{\bf y})\mid {\bf y} - y = 2x ({\bf x} -x)\}$. 
Indeed, one may argue as in the proof of Lemma~\ref{main_under_tangent_Le} to show that~\eqref{eq280101} 
yields~\eqref{roof_problem_dynamic}.

The right hand side of~\eqref{eq280101} is linear with respect to $\xx$ when $(\xx,\yy) \in \ell$ and is equal to
$$
L(\xx)=\mmm(x,y)+\frac{\partial\mmm}{\partial x}(x,y)\cdot(\xx-x)+
\frac{\partial\mmm}{\partial y}(x,y)\cdot2x(\xx-x)
$$
since $G = \mmm$ on $\OmSmile$. Lemma~\ref{VerifMainIneq} implies that~\eqref{eq280101} holds true for 
$(\xx,\yy) \in \OmSmile \cap \ell$ because $G(\xx,\yy) = \mmm(\xx,\yy)$. The point 
$(\frac{x+\sqrt{2-x^2}}{2}, 1+x\sqrt{2-x^2})$ is the intersection of $\ell$ with the common boundary of~$\OmR$ and $\OmSmile$, therefore, $L(\frac{x+\sqrt{2-x^2}}{2}) =1$ by the construction of~$\mmm$ 
(see~\eqref{grad_direction}). Also, we know that $L(x) = \mmm(x,y) < 1$, therefore $L(\xx)\geq 1$ 
for $(\xx,\yy) \in \OmR\cap \ell$.   

Thus, it remains  to prove that~\eqref{eq280101} holds for $(\xx,\yy) \in \OmL\cap \ell$:
\eq{\label{eq280102} 
G(\xx,\yy)\leq L(\xx ) = 1+\frac{2\xx-(x+\sqrt{2-x^2})}{x-\sqrt{2-x^2}}\big(G(x,x^2+1)-1\big). 
}
Recall that $G(\xx, \yy) = \bbb{\xx}{\yy}$ is given by~\eqref{b_eps_def} for $(\xx,\yy) \in \OmL$:
\eq{\label{eq280103}
G(\xx,\yy)=\frac12e^{1+\frac{\xx}{\sqrt{\yy-\xx^2}}}=\frac12e^{1+\frac{\xx}{\sqrt{1-(x-\xx)^2}}}, 
}
here we have used that $\yy = 2x (\xx-x) + x^2+1$. The value $G(x,x^2+1)$ equals to $\Psi(x)$ defined 
in~\eqref{psi_def}: 
\begin{equation}\label{eq280104}
G(x,x^2+1) = \Psi(x) = 1-\frac{\sqrt{2-x^2}- x}{4}e^{-\arcsin\left(\frac{x}{\sqrt{2}}\right)-\frac{\pi}{4}}.
\end{equation}
We rewrite~\eqref{eq280102} using~\eqref{eq280103} and~\eqref{eq280104}:
\eq{\label{eq280105}
\frac{1}{2}e^{1 + \frac{\xx}{\sqrt{1-(x-\xx)^2}}} \leq 
1 + \frac{2\xx - (x+\sqrt{2-x^2})}{4}  e^{-\arcsin\left(\frac{x}{\sqrt{2}}\right)-\frac{\pi}{4}}.
}
We introduce the variables 
\eq{
\alpha = \frac{\pi}{4} + \arcsin{\frac{x}{\sqrt{2}}}, \qquad \gamma = \arcsin(x-\xx).
}
Then,
\begin{gather*}
x = \sqrt{2}\sin\left(\alpha - \frac{\pi}{4}\right) = \sin \alpha - \cos \alpha,
\\
\xx = x-\sin\gamma = \sin \alpha - \cos \alpha-\sin\gamma,
\\
x+\sqrt{2-x^2}=\sqrt{2}\sin\left(\alpha-\frac{\pi}{4}\right)+\sqrt2\cos\left(\alpha-\frac\pi4\right)=2\sin\alpha.
\end{gather*}
Note that $\alpha,\gamma \in [0,\frac{\pi}{2}]$. The condition $(\xx,\yy) \in \OmL$ implies that 
$\xx \leq \frac{x-\sqrt{2-x^2}}{2}$, i.\,e., 
$x - \xx \geq \frac{x+\sqrt{2-x^2}}{2}$. From this we obtain $0 \leq \alpha \leq \gamma \leq \frac{\pi}{2}$. 
Rewrite~\eqref{eq280105} in the variables $\alpha, \gamma$:
\eq{
\frac12e^{1+\frac{\sin\alpha-\cos\alpha-\sin\gamma}{\cos\gamma}}\leq1-\frac{\cos\alpha+\sin\gamma}2e^{-\alpha}.
}

It suffices to show that for the fixed parameter $\gamma$, $0 \le \gamma\le\frac\pi2$, the function
\begin{equation*}
F(\alpha)=e^{1+\frac{\sin\alpha-\cos\alpha-\sin\gamma}{\cos\gamma}}-2+
\cos\alpha e^{-\alpha}+\sin{\gamma}e^{-\alpha}
\end{equation*}
attains only non-positive values for $0\le\alpha\le\gamma.$
Our next step is to prove the convexity of $F$. Its first and second derivatives are written below
\begin{equation*}
F'(\alpha) = e^{-\alpha}\left(-\sin{\alpha}-\cos\alpha - \sin\gamma\right) + 
e^{1+\frac{\sin\alpha - \cos\alpha - \sin\gamma}{\cos\gamma}}
\left(\frac{\cos\alpha + \sin\alpha}{\cos\gamma}\right);
\end{equation*}
\begin{equation*}
F''(\alpha) = e^{-\alpha}\left(2\sin \alpha + \sin\gamma\right) + 
e^{1+\frac{\sin\alpha-\cos\alpha-\sin\gamma}{\cos \gamma}}
\left(\frac{\cos\alpha+\sin\alpha}{\cos\gamma}\right)^2 
\end{equation*}
\begin{equation*}
+ e^{1+\frac{\sin\alpha-\cos\alpha-\sin\gamma}{\cos\gamma}}\left(\frac{\cos\alpha-\sin\alpha}{\cos\gamma}\right).
\end{equation*}
The first term on the right side in the last equality is non-negative. 
By grouping the second and third terms, we get that the second derivative of the function $F$ is 
also always non-negative. Indeed, it follows from the estimate $1\ge\sin\alpha\cos\gamma$ and that all 
other expressions involved are  positive. We have shown that the function $F$ is convex. 
To estimate its values from above on $[0,\gamma]$, it suffices to show  $F(0) \le 0$ and $F(\gamma) \le 0$.
We start with the case $\alpha = 0$:
\begin{equation*}
F(0) = (1+\sin \gamma) + e^{1-\frac{1+\sin\gamma}{\cos\gamma}} - 2.
\end{equation*}
The statement we need to prove is equivalent to the fact that the function
\begin{equation*}
\fff(\gamma) = 1 - \frac{1+\sin \gamma}{\cos \gamma} - \log(1-\sin \gamma)
\end{equation*}
takes only non-positive values when $\gamma \in [0,\pi /2]$. It should be noted that $\fff(0) = 0$, while
 \begin{equation*}
\fff'(\gamma)=-\frac{\cos^2\gamma+\sin\gamma(1+\sin\gamma)}{\cos^2\gamma}+\frac{\cos\gamma}{1-\sin\gamma} 
= \frac{(\cos \gamma -1)(1+\sin \gamma)}{\cos^2\gamma}\le 0.
\end{equation*}
We have obtained that $\fff(\gamma) \le 0$, so the estimate $F(0) \le 0$ follows.
Now, we verify the inequality for the right endpoint of the segment, i.\,e., for $\alpha = \gamma$:
 \begin{equation*}
F(\gamma) = e^{-\gamma}(\sin\gamma+\cos\gamma) - 1 \le 0.
\end{equation*}
One may easily see that for $\gamma = 0$ the inequality turns into equality. 
Taking the derivative of the function on the left side of this inequality, we get~$-2\sin\gamma e^{-\gamma}$.
The last term is negative for  $\gamma \in (0, \frac{\pi}{2}]$ and the estimate $F(\gamma) \le 0$ follows.

\subsection{Computation of the constant}\label{s44_sharp}
In this section, we present the proof of Corollary~\ref{cor_sharp_tail}.
Recall that our goal is to find the optimal constant in~\eqref{tail_sq}. Since the square function is 
homogeneous and vanishes on constants, it suffices to find the best possible constant~$c_{opt}$ in the estimate
\begin{equation}\label{translated_ineq}
P(\varphi_\infty\geq0)\leq c e^{-\lambda},
\qquad \text{  where  }\ \varphi_0 =-\lambda\ \text{ and }\ \|S\varphi\|_{L^\infty} = 1.
\end{equation}
Recall that the Bellman function $\Bell(x,y,z)$ was defined by formula \eqref{Bellman}, and in 
Lemma~\ref{SimpleMajorant} we have shown that the inequality $\Bell(x,y,z)\leq b_{\sqrt{1-z^2}}(x,y)$ 
is true. Thus, the optimal constant $c_{opt}$ may be estimated as follows
\begin{equation}\label{copt}
c_{opt} = \sup\big\{e^\lambda\Bell(-\lambda,y,0)\mid \lambda\in\R,\,\lambda^2\leq y\leq\lambda^2+1\big\}\leq
\sup\big\{e^{-x}b_1(x,y)\mid (x,y)\in\omega_1\big\}\,.   
\end{equation}
Recall that we have split the domain $\omega_1$ into the subdomains $D^1_1, D^1_2$, $D^1_3$, $D^1_4$, and 
the function $b_{1}(x,y)$ was defined on them by~\eqref{b_eps_def}. We will continue the argument 
by the estimation of $s(x,y):=e^{-x}b_1(x,y)$ in each subdomain.

Clearly, $s(x,y) \le 1$, when $(x,y) \in D_1^1$. 

For $(x,y) \in D^1_2$, we have
$$
s(x,y) = e^{-x} \left(1-\frac{y-2x}{8}\right) \le e^{-x} \left(1-\frac{|x|-x}{4}\right),
$$
since $y \ge 2|x|$. The function on the right hand side decreases on $[-1,1]$ and takes value 
$\frac{e}{2}$ at $-1$, therefore $s(x,y)\leq \frac{e}{2}$ on $D_2^1$. 

Next, consider $(x,y) \in D^1_3$. The relations $y \le -2 x$ and $-2 \le x \le 0$ imply
$$
s(x,y) = \left(1 - \frac{x^2}{y}\right)e^{-x} \le \left(1 + \frac{x}{2} \right)e^{-x} \le \frac{e}{2}.
$$

Finally, we take $(x,y) \in D^1_4$ and set $t = \sqrt{1 -y + x^2}$ to get
\begin{equation}\label{EstimateOnD4}
s(x,y) = \frac{e}{2} (1-t)e^t \le \frac{e}{2}, 
\end{equation}
since $t \in [0,1].$ Thus, we have proved~$c_{opt}\le \frac{e}{2}$. 

Now we notice that for $x=-1$ and $y=2$ we have $\Bell(-1,2,0)=b_1(-1,2)=\frac12$, 
and therefore,~\eqref{copt} implies that $c_{opt}\geq\frac{e}2$, which means $c_{opt}=\frac{e}2$.

\bibliography{mybib}{}
\bibliographystyle{amsplain}

St. Petersburg State University, Department of Mathematics and Computer Science.

d.m.stolyarov@spbu.ru

vasyunin@pdmi.ras.ru

pavelz@pdmi.ras.ru

i.zlotnikov@spbu.ru
\end{document}